\DeclareSymbolFont{AMSb}{U}{msb}{m}{n}
\documentclass[noamsfonts,10pt]{amsart}
\usepackage{amsrefs}
\usepackage[colorlinks, linkcolor=blue, citecolor=blue, urlcolor=blue]{hyperref}
\usepackage{amsthm}
\usepackage{mathtools}
\usepackage{array, longtable}
\usepackage[table]{xcolor}
\usepackage{caption}
\usepackage{subcaption}
\usepackage{fullpage}
\usepackage{enumitem}
\usepackage{ytableau}
\usepackage{tikz}
\usetikzlibrary{calc}
\usetikzlibrary{decorations.pathreplacing}
\usepackage[indent]{parskip}
\usepackage{cancel}
\usepackage[charter]{mathdesign}

\tikzstyle{dot}=[circle,fill=black, minimum size = 2.25pt, inner sep=0pt]

\tikzstyle{corner}=[circle,draw=black,fill=none, minimum size = 5pt, inner sep=0pt]

\newcommand{\gridbox}[2]{%
    \draw [lightgray,fill=lightgray] ($(#1,#2) + (-0.5,-0.5)$) rectangle ($(#1,#2) + (0.5,0.5)$);
}

\newcommand{\ijBox}[2]{%
    \draw [fill=lightgray!30] ($(#1,#2) + (-0.5,-0.5)$) rectangle ($(#1,#2) + (0.5,0.5)$);
}

\newcommand{\rshade}{*(lightgray)}

\newcommand{\exIntro}[1]{%
    \begin{tikzpicture}[scale=.2,baseline={([yshift=-2pt]current bounding box.center)}]
        \foreach \x/\y in {#1} {
        \gridbox{\x}{\y}
}
        \foreach \x in {1,...,3}{\foreach \y in {1,...,3}{\node [dot] at (\x,\y) {};}}
    \end{tikzpicture}
    }

\newcommand{\N}{\mathbb{N}}
\newcommand{\R}{\mathbb{R}}
\newcommand{\NN}{\mathcal{N}}
\newcommand{\M}{\mathcal{M}}
\renewcommand{\P}{\mathcal{P}}
\newcommand{\D}{\mathcal{D}}
\newcommand{\DZ}{\mathcal{D}^{\mathrm{Z} \!\!\!\! \backslash}}
\newcommand{\al}{\alpha}
\newcommand{\be}{\beta}
\DeclareMathOperator{\Fr}{Frame}
\DeclareMathOperator{\Pic}{Pict}
\renewcommand{\O}{\mathcal{O}}
\newcommand{\F}{\mathcal{F}}

\newcommand{\I}{\mathcal{I}}
\newcommand{\NE}[1]{\operatorname{NE}(#1)}
\newcommand{\SW}[1]{\operatorname{SW}(#1)}
\renewcommand{\k}{\Bbbk}
\renewcommand{\d}{\operatorname{d}}
\newcommand{\dd}{\operatorname{diag}}
\newcommand{\vv}{\operatorname{vert}}
\newcommand{\ff}{\operatorname{fact}}
\newcommand{\gz}{\textcolor{lightgray}{0}}
\DeclareMathOperator{\Vol}{Vol}
\DeclareMathOperator{\Ehr}{Ehr}
\DeclareMathOperator{\Hilb}{Hilb}
\DeclareMathOperator{\Kdim}{Kdim}
\DeclareMathOperator{\cone}{cone}
\DeclareMathOperator{\supp}{supp}
\newcommand{\res}{\operatorname{res}}
\DeclareMathOperator{\cor}{cor}
\DeclareMathOperator{\Par}{Par}
\DeclareMathOperator{\Conn}{Conn}
\newcommand{\sA}{\mathsf{A}}
\newcommand{\sB}{\mathsf{B}}
\newcommand{\sC}{\mathsf{C}}
\newcommand{\sD}{\mathsf{D}}
\newcommand{\1}{\mathbf{1}}
\newcommand{\mat}[1]{\left[ \begin{smallmatrix} #1 \end{smallmatrix} \right]}

\theoremstyle{plain}
\newtheorem{theorem}{Theorem}[section]
\newtheorem{lemma}[theorem]{Lemma}
\newtheorem{corollary}[theorem]{Corollary}
\newtheorem{proposition}[theorem]{Proposition}

\newtheorem{maintheorem}{Theorem}

\theoremstyle{definition}
\newtheorem{dfn}[theorem]{Definition}
\newtheorem{remark}[theorem]{Remark}
\newtheorem{example}[theorem]{Example}
\newtheorem{construction}[theorem]{Construction}
\numberwithin{equation}{section}

\title{A lattice path model for the volume of the Monge polytope}

\author{William Q.~Erickson}
\address[Erickson]{
University of Tennessee at Chattanooga \\
615 McCallie Avenue \\
Chattanooga, TN 37403} 
\email{william-erickson01@utc.edu}

\author{Nicholas B.~Jones}
\address[Jones]{University of North Texas\\
1155 Union Circle\\
Denton, TX 76203} 
\email{nicholas.jones@unt.edu}

\begin{document}

\begin{abstract}
    Monge matrices arise throughout combinatorial optimization and algorithm design;
    the \emph{Monge polytope} $\M_{pq}$ is the set of $p \times q$ Monge matrices lying inside the standard simplex on the set of matrix coordinates.
    We find a Stanley decomposition of the associated affine semigroup, and use it to obtain a volume formula for $\M_{pq}$ expressed as a sum over ``Z-avoiding'' Delannoy paths in a $p \times q$ grid.
    An efficient dynamic-programming implementation of this formula computes the volume in dimensions far beyond the reach of general-purpose exact-volume algorithms (e.g., the volume of $\M_{20,20}$, which has dimension 399, is computed in a fraction of a second).
    As a corollary of our Stanley decomposition, we also obtain a combinatorial closed form for the Ehrhart series of $\M_{pq}$.
\end{abstract}








\maketitle

\setcounter{tocdepth}{1}

\renewcommand{\baselinestretch}{0.5}\normalsize
\tableofcontents
\renewcommand{\baselinestretch}{1.0}\normalsize

\section{Introduction}

\subsection*{The Monge polytope}

A \emph{Monge matrix} is a real matrix $M$ satisfying the condition 
\begin{equation}
    \label{Monge condition intro}
    M_{ij} + M_{IJ} \leq M_{iJ} + M_{Ij} \text{ for all $i < I$ and $j < J$}.
\end{equation}
Monge matrices have their roots in Gaspard Monge's foundational 1781 work~\cite{Monge} in optimal transport.
Monge's motivation was the linear programming problem known today as the \emph{transportation problem} (having since been reformulated by Kantorovich, Hitchcock, and Koopman):
find the minimum-cost plan for transporting mass from $p$ supply sites to $q$ demand sites, where the pairwise supply--demand costs are given by a $p \times q$ cost matrix~$M$.
Hoffman~\cite{Hoffman} showed that the transportation problem can be solved by a simple greedy algorithm (the ``northwest corner rule''), for every choice of supply and demand vectors, if and only if the cost matrix $M$ is Monge.
Essentially, then, the Monge condition~\eqref{Monge condition intro} precisely describes the cost structures for which the transportation problem is trivial to solve.
Beyond their origin in transportation, Monge matrices (and the related Monge sequences~\cite{BN22}) play a key role across combinatorial optimization (see the excellent survey~\cite{Burkard}) and algorithm design~\cites{Aggarwal87,FR06,GP,KMW10}.

The existing literature treats the Monge condition~\eqref{Monge condition intro} largely as an algorithmic hypothesis, and asks relatively little about the geometry of the set of matrices satisfying it.
It is well known that the set of all nonnegative $p \times q$ Monge matrices forms a pointed cone in $\R^{p \times q}$, whose extremal rays were described by Rudolf and Woeginger~\cite{RW95}.
However, in the context of the transportation problem, the full cone is actually redundant, in the sense that the set of optimal solutions remains unchanged when one simply scales the cost matrix $M$.
This is because an optimal solution to the transportation problem is a matrix $T$ that minimizes the Hadamard product $\sum_{i,j} M_{ij} T_{ij} = \operatorname{tr}(M^t T)$, and this product is clearly linear in $T$.
In other words, in the cone of Monge matrices, each ray through $0$ should really be viewed as a single equivalence class, since all the cost matrices on that ray have the same set of optimal solutions.
Normalizing so that the entries of $M$ sum to $1$ selects a canonical representative from each equivalence class ---
geometrically, the hyperplane $\sum_{i,j} M_{ij} = 1$ meets each ray in exactly one point --- cutting out the bounded polytope we call the \emph{Monge polytope}
\[
    \M_{pq} \coloneqq \Big\{ \text{Monge matrices in $\R^{p \times q}$ with nonnegative entries summing to $1$} \Big\},
\]
sitting inside the standard simplex $\Delta_{pq-1}$ on the set of matrix coordinates.

Despite the ubiquity of Monge matrices, relatively little is known about the geometry of the Monge polytope.
(What we call the Monge polytope $\M_{pq}$ in this paper is quite different from the polytopes studied in~\cite{Friesecke} and~\cite{Vogler} and the references therein;
those authors consider polytopes whose points are solutions, rather than cost functions, for transport problems.
Such polytopes are generally known as \emph{transportation polytopes}.)
The most obvious goal is to find a closed formula for its normalized volume
\begin{equation}
    \label{volume in intro}
    \Vol(\M_{pq})
 \coloneqq \text{volume of $\M_{pq}$ relative to $\Delta_{pq-1}$},
 \end{equation}
 thus answering the question of how ``rare'' the Monge property~\eqref{Monge condition intro} is among $p \times q$ matrices in the standard simplex.
 The next natural goal --- especially for those readers familiar with nonnegative integer matrix enumeration problems in the vein of Stanley's book~\cite{StanleyCCA} --- is to solve the discrete version of the problem, namely finding a closed form for the Ehrhart series
\begin{equation}
    \label{Ehrart series intro}
    \Ehr(\M_{pq}; z) \coloneqq \sum_{t=0}^\infty \left| \Big\{ \text{Monge matrices in $\N^{p \times q}$ with entries summing to $t$} \Big\} \right| \, z^t.
\end{equation}
The main result of this paper is a combinatorial model that yields explicit formulas for~\eqref{volume in intro} and~\eqref{Ehrart series intro}, to be previewed below in Theorems~\ref{thm:volume in intro} and~\ref{thm:Ehrhart in intro}.
This fits within a broader tradition of finding explicit combinatorial models for polytope volumes, such as flow polytopes in~\cite{Benedetti19} or the Birkhoff polytope in~\cite{DLY09}.

\subsection*{Main result: volume formula}

Our volume formula is expressed in terms of \emph{Delannoy paths}, that is, paths $(1,1) \rightarrow (p,q)$ using steps to the south ($\downarrow$), east $(\rightarrow)$, or southeast $(\searrow)$ in matrix coordinates.
We say that a Delannoy path $D$ is \emph{Z-avoiding} if it does not contain a ``Z''-shaped subpath, that is, $\rightarrow \downarrow \cdots \downarrow\rightarrow$.
A \emph{vertical component} of $D$ is connected by a run of $\downarrow$'s that is not adjacent to a $\rightarrow$.
A path $D$ breaks the $p \times q$ grid into two connected components whose shapes are given by integer partitions $\al(D)$ and $\be(D)$;
in expressing the theorem below, we recall that for a partition $\al = (\al_1, \ldots, \al_\ell)$, it is standard to write $\al! \coloneqq \al_1! \cdots \al_\ell!$, and to write $\al'$ for the conjugate partition obtained by transposing the Young diagram of $\al$.

\begin{maintheorem}[see Theorem~\ref{thm:volume in body}]
    \label{thm:volume in intro}
    The normalized volume~\eqref{volume in intro} of the Monge polytope is given by
    \[
        \Vol(\M_{pq}) = \frac{1}{p^q q^{p-1}} \sum_{D \in \DZ_{pq}} q^{\dd(D)} \: \frac{\vv(D)}{\ff(D)},
    \]
    where
    \begin{itemize}
        \item[] $\DZ_{pq}$ is the set of Z-avoiding Delannoy paths from $(1,1)$ to $(p,q)$;
        \item[] $\dd(D)$ is the number of $\searrow$'s in $D$;
        \item[] $\vv(D)$ is the product of sizes of the vertical components in $D$;
        \item[] $\ff(D) \coloneqq \al! \al'! \be! \be'!$, with $\al = \al(D)$ and $\be = \be(D)$ as described above.
    \end{itemize}
\end{maintheorem}

To briefly illustrate Theorem~\ref{thm:volume in intro}, take $p=q=3$:
\[
{
\renewcommand{\arraystretch}{1.5}
\begin{array}{c|c|c|c|c|c|c|c|c|c|c}
    D \in \D_{pq}
    & 
    \begin{tikzpicture}[scale=.2]
        \foreach \x in {1,...,3}{\foreach \y in {1,...,3}{\node [dot] at (\x,\y) {};}}
        \draw[very thick] (1,3) --++ (0,-2) -- ++(2,0);
    \end{tikzpicture}
    & 
    \begin{tikzpicture}[scale=.2]
        \foreach \x in {1,...,3}{\foreach \y in {1,...,3}{\node [dot] at (\x,\y) {};}}
        \draw[very thick] (1,3) --++ (0,-1) -- ++(1,-1) -- ++(1,0);
    \end{tikzpicture}
    & 
    \begin{tikzpicture}[scale=.2]
        \foreach \x in {1,...,3}{\foreach \y in {1,...,3}{\node [dot] at (\x,\y) {};}}
        \draw[very thick] (1,3) --++ (0,-1) -- ++(1,0) -- ++(1,-1);
    \end{tikzpicture}
    & 
    \begin{tikzpicture}[scale=.2]
        \foreach \x in {1,...,3}{\foreach \y in {1,...,3}{\node [dot] at (\x,\y) {};}}
        \draw[very thick] (1,3) --++ (0,-1) -- ++(2,0) -- ++(0,-1);
    \end{tikzpicture}
    & 
    \begin{tikzpicture}[scale=.2]
        \foreach \x in {1,...,3}{\foreach \y in {1,...,3}{\node [dot] at (\x,\y) {};}}
        \draw[very thick] (1,3) --++ (1,0) -- ++(0,-1) -- ++(1,-1);
    \end{tikzpicture}
    & 
    \begin{tikzpicture}[scale=.2]
        \foreach \x in {1,...,3}{\foreach \y in {1,...,3}{\node [dot] at (\x,\y) {};}}
        \draw[very thick] (1,3) --++ (1,0) -- ++(1,-1) -- ++(0,-1);
    \end{tikzpicture}
    & 
    \begin{tikzpicture}[scale=.2]
        \foreach \x in {1,...,3}{\foreach \y in {1,...,3}{\node [dot] at (\x,\y) {};}}
        \draw[very thick] (1,3) --++ (2,0) -- ++(0,-2);
    \end{tikzpicture}
    & 
    \begin{tikzpicture}[scale=.2]
        \foreach \x in {1,...,3}{\foreach \y in {1,...,3}{\node [dot] at (\x,\y) {};}}
        \draw[very thick] (1,3) --++ (1,-1) -- ++(0,-1) -- ++(1,0);
    \end{tikzpicture}
    & 
    \begin{tikzpicture}[scale=.2]
        \foreach \x in {1,...,3}{\foreach \y in {1,...,3}{\node [dot] at (\x,\y) {};}}
        \draw[very thick] (1,3) --++ (1,-1) -- ++(1,0) -- ++(0,-1);
    \end{tikzpicture}
    &
    \begin{tikzpicture}[scale=.2]
        \foreach \x in {1,...,3}{\foreach \y in {1,...,3}{\node [dot] at (\x,\y) {};}}
        \draw[very thick] (1,3) --++ (1,-1) -- ++(1,-1);
    \end{tikzpicture}\\ \hline
    \dd(D) & 0 & 1 & 1 & 0 & 1 & 1 & 0 & 1 & 1 & 2 \\
    \vv(D) & 1 & 2 & 1 & 1 & 1 & 2 & 1 & 1 & 1 & 1  \\
    \ff(D) & 16 & 16 & 8 & 4 & 8 & 16 & 16 & 8 & 8 & 16 \\ \hline
    q^{\dd(D)} \frac{\vv(D)}{\ff(D)} & \frac{1}{16} & \frac{6}{16} & \frac{3}{8} & \frac{1}{4} & \frac{3}{8} & \frac{6}{16} & \frac{1}{16} & \frac{3}{8} & \frac{3}{8} & \frac{9}{16}
\end{array}
}
\]
Since the prefactor $p^q q^{p-1} = 3^3 \cdot 3^2 = 243$, we obtain
\[
    \Vol(\M_{3,3}) = \frac{1}{243} \left( \frac{1}{16} + \frac{6}{16} + \frac{3}{8} + \frac{1}{4} + \frac{3}{8} + \frac{6}{16} + \frac{1}{16} + \frac{3}{8} + \frac{3}{8} + \frac{9}{16} \right) = \frac{17}{1296}.
\]
Far beyond toy examples like this, however, Theorem~\ref{thm:volume in intro} affords an enormous computational advantage compared to direct algorithms for polytope volume, which use triangulation or related decomposition methods, and are notoriously expensive~\cite{GK}.
The cost of exact triangulation-based volume algorithms scales badly (typically exponentially) in the \emph{dimension} of the polytope;
in practice, this quickly becomes prohibitive even at relatively small dimension.
The 2020 benchmarking study~\cite{Enge} of exact volume algorithms, reproducing B\"ueler--Enge--Fukuda's original study~\cite{BEF}, reports timings only for polytopes of dimension up to 15, and even then, many of the algorithms exceeded memory or were ruled out as impractical~\cite{Enge}*{Table 1}.
Famously, the volume of the 10th Birkhoff polytope $\mathcal{B}_{10}$ (dimension 81) required nearly 17 years of runtime to compute, scaled to a 1GHz processor~\cite{BeckPixton03}*{p.~634}.

By contrast, the formula in Theorem~\ref{thm:volume in intro} lends itself easily to a standard dynamic-programming/transfer-matrix technique over lattice paths;
Mathematica code for this is available at the following link:
\begin{center}
    \url{https://github.com/WilliamQErickson/Monge-volume}
\end{center}
The dynamic program above sweeps the $p\times q$ grid once rather than enumerating Delannoy paths;
at each point in the grid, the program maintains a running  total for every possible local state (with respect to the three statistics in Theorem~\ref{thm:volume in intro}) which a path could have upon reaching that point.
We verified our formula in Theorem~\ref{thm:volume in intro} by checking it against brute-force triangulation methods in Sage, for the following $p$- and $q$-values (assuming $p \leq q$, since the volume of $\M_{pq}$ is symmetric in $p$ and $q$).
In the $p=q$ case, the authors' laptop computers seem unable to handle the brute-force computation beyond $p=q=6$.
\begingroup
\renewcommand{\arraystretch}{1.5}
\[
    \begin{array}{|>{\columncolor{lightgray}}c |c|c|c|c|c|c|c|}
    \hline
    \rowcolor{lightgray} p \backslash q & 1 & 2 & 3 & 4 & 5 & 6  \\ \hline
    1 & 1 & 1 & 1 & 1 & 1 & 1  \\ \hline
    2 & - & \frac{1}{2} & \frac{1}{6} & \frac{1}{24} & \frac{1}{120} & \frac{1}{720} \\ \hline
    3 & - & - & \frac{17}{1296} & \frac{37}{62208} & \frac{163}{9331200} & \frac{241}{671846400} \\ \hline
    4 & - & - & - & \frac{361}{95551488} & \frac{3623}{286654464000} & \frac{6163}{247669456896000} \\ \hline
    5 & - & - & - & - & \frac{84341}{21499084800000000} & \frac{336641}{557256278016000000000} \\ \hline
    6 & - & - & - & - & - & \frac{57455963}{9359765606561218560000000000} \\ \hline
    \end{array}
\]
\endgroup
By contrast, using the dynamic programming implementation of Theorem~\ref{thm:volume in intro} in the link above, we computed $\Vol(\M_{20,20})$ in a fraction of a second (despite $\M_{20,20}$ having dimension 399).
This underscores the fact that combinatorial methods exploiting the specific structure of a polytope family can reach regimes that are, at present, entirely out of reach for general-purpose exact-volume software.
We note, however, that such gains are not automatic: 
De Loera, Liu, and Yoshida~\cite{DLY09} found an elegant combinatorial formula for the volume of the Birkhoff polytope $\mathcal{B}_n$, expressed as a sum over permutations and rooted directed trees, but the number of these terms grows rapidly enough with $n$ that the formula does not outperform other methods for large $n$.
The efficiency of our formula instead relies on the fact that the terms of Theorem~\ref{thm:volume in intro} are indexed by a comparatively small\footnote{For the generating function that counts Z-avoiding Delannoy paths, see Proposition~\ref{prop:count Dpq} and the table in~\eqref{table Dpq}.} 
and highly structured set of
combinatorial objects (Z-avoiding Delannoy paths), reflecting a monotonic structure specific to the Monge condition.

\subsection*{Main technique: Stanley decompositions}

Our primary technical contribution (and the source of Theorem~\ref{thm:volume in intro}) is Theorem~\ref{thm:Stanley decomp}, giving an explicit Stanley decomposition of the affine semigroup $S(\M_{pq})$ associated to the Monge polytope.
A \emph{Stanley decomposition} of a commutative monoid $S$ is a finite disjoint union
\[
    S = \coprod_{i \in I} \: (a_i + \N B_i), \qquad a_i \in S, \text{ with $B_i \subset S$ linearly independent}.
\]
Thus every element $s \in S$ has a unique representation $s = a_i + \sum_{b \in B_i} n_b b$, where $i \in I$ and each $n_b \in \N$ (see Example~\ref{ex:sigma}).
Our Stanley decomposition of $S(\M_{pq})$ in Theorem~\ref{thm:Stanley decomp} is indexed by \emph{osculating pairs} $(\al, \be)$ of integer partitions, together with the facets of an associated shellable simplicial complex $\Delta(\al, \be)$; see Definition~\ref{def:Parpq and Opq}.
While this decomposition is of independent interest, we constructed it primarily to yield a combinatorial model for the Ehrhart series of $\M_{pq}$ (given below in Theorem~\ref{thm:Ehrhart in intro}, which in the body of the paper is a direct corollary of Theorem~\ref{thm:Stanley decomp}), and ultimately for the volume formula in Theorem~\ref{thm:volume in intro} as well.

\begin{maintheorem}[see Corollary~\ref{cor:Ehrhart series}]
    \label{thm:Ehrhart in intro}
    The Ehrhart series~\eqref{Ehrart series intro} of the Monge polytope is given by
    \[
    \Ehr(\M_{pq}; z) =
    \frac{1}{(1-z^p)^q} \cdot 
    \sum_{(\al, \be) \in \O_{pq}} \frac{z^{\d(\al, \be)} }{\prod_{(i,j) \in \al} (1-z^{ij}) \prod_{(i,j) \in \be} (1-z^{ij})} 
    \left( 
    \sum_{F \in \F(\al,\be)} \frac{z^{q \left| \res(F) \right|} }{ (1-z^q)^{|F|}} \right),
    \]
    where
    \begin{itemize}
        \item[] $\O_{pq}$ is the set of ``osculating pairs'' $(\al,\be)$ of integer partitions, given in Definition~\ref{def:Parpq and Opq}(\ref{osculating});
        \item[] $\d( \al, \be)$ is the statistic given in Definition~\ref{def:Parpq and Opq}(\ref{d definition});
        \item[] $\F(\al,\be)$ consists of certain sets $F$ where $\res(F) \subseteq F \subseteq \{1, \ldots, p\}$, as given by Construction~\ref{const:F and R}.
    \end{itemize}
\end{maintheorem}

Just as we did for Theorem~\ref{thm:volume in intro}, we have written out the details for Theorem~\ref{thm:Ehrhart in intro}, term by term, in the special case $p=q=3$; due to its length, we refer the reader to Appendix~\ref{appendix:3 by 3}.
As a preview, using Theorem~\ref{thm:Ehrhart in intro} and taking the power series expansion yields
\[
    \Ehr(\M_{3,3}; z) = 1 + 2 z + 7 z^2 + 18 z^3 + 41 z^4 + 86 z^5 + 176 z^6 + 325 z^7 + 
 587 z^8 + 1016 z^9 + 1686 z^{10} + \cdots.
\]
See below for the concrete verification of the first few terms:
\begingroup
\renewcommand{\arraystretch}{1.5}
\[
\begin{array}{c|l|c}
    t & \Big\{\text{Monge matrices in $\N^{3 \times 3}$ with entries summing to $t$}\Big\} & \text{Count} \\ \hline
    0 & \left\{\mat{0&0&0\\0&0&0\\0&0&0} \right\} & 1 \\[2ex]
    1 & \left\{ \mat{0&0&1\\0&0&0\\0&0&0}, \mat{0&0&0\\0&0&0\\1&0&0} \right\} & 2 \\[2ex]
    2 & \left\{ \mat{0&1&1\\0&0&0\\0&0&0}, \mat{0&0&2\\0&0&0\\0&0&0}, \mat{0&0&1\\0&0&1\\0&0&0}, \mat{0&0&1\\0&0&0\\1&0&0}, \mat{0&0&0\\1&0&0\\1&0&0}, \mat{0&0&0\\0&0&0\\2&0&0}, \mat{0&0&0\\0&0&0\\1&1&0} \right\} & 7 \\[2ex]
    3 & \left\{ \begin{array}{l} \mat{1&1&1\\0&0&0\\0&0&0},
\mat{1&0&0\\1&0&0\\1&0&0},
\mat{0&1&2\\0&0&0\\0&0&0},
\mat{0&1&1\\0&0&0\\1&0&0},
\mat{0&1&0\\0&1&0\\0&1&0},
\mat{0&0&3\\0&0&0\\0&0&0},
\mat{0&0&2\\0&0&1\\0&0&0},
\mat{0&0&2\\0&0&0\\1&0&0},
\mat{0&0&1\\1&0&0\\1&0&0}, \\[1ex]
\mat{0&0&1\\0&0&1\\1&0&0},
\mat{0&0&1\\0&0&1\\0&0&1},
\mat{0&0&1\\0&0&0\\2&0&0},
\mat{0&0&1\\0&0&0\\1&1&0},
\mat{0&0&0\\1&1&1\\0&0&0},
\mat{0&0&0\\1&0&0\\2&0&0},
\mat{0&0&0\\0&0&0\\3&0&0},
\mat{0&0&0\\0&0&0\\2&1&0},
\mat{0&0&0\\0&0&0\\1&1&1}
\end{array}
\right\} & 18
\end{array}
\]
\endgroup

We close this introduction with a more detailed remark on method.
Our main results arose from a two-step template:
first exhibit a Stanley decomposition of the affine semigroup, and then use it to derive a combinatorial interpretation for the Hilbert series~\eqref{Hilb k[S]} and the multiplicity~\eqref{e(k[S])} of the semigroup ring.
(The Hilbert series falls directly out of a Stanley decomposition, but in order to obtain a multiplicity formula,  one must carefully describe the Stanley spaces of maximal Krull dimension.) 
In the polytope setting of the present paper, this template takes the concrete form of an Ehrhart series~\eqref{Ehrart series intro} and a volume~\eqref{volume in intro}, respectively, but the underlying strategy is not specific to polytope problems. 
In unrelated recent joint work with M.~Hunziker~\cites{EHJellyfish,EHBernstein}, we used the identical template in the context of unitary highest weight modules (where the multiplicity is known as the ``Bernstein degree'').
In fact, the combinatorial model in those two papers also involved lattice paths. 
It is tempting to wonder whether this Stanley decomposition template might offer a useful strategy in other settings where a Hilbert series or multiplicity is otherwise hard to write down; we hope to explore this further elsewhere.

\section{Preliminaries}
\label{sec:preliminaries}

In this section, we record some standard facts from combinatorial commutative algebra which we will need in the rest of the paper.

\subsection*{Semigroup rings and Hilbert series}

We largely follow the exposition in~\cite{MillerSturmfels}*{Ch.~7}.
Throughout the paper we write $\N$ for the set of nonnegative integers.
Let $S$ be a finitely generated commutative monoid.
(Following standard usage in this area, for example Bruns--Herzog~\cite{BrunsHerzog} and Miller--Sturmfels~\cite{MillerSturmfels}), we will also use the term ``semigroup'' for what is, strictly speaking, a commutative monoid, that is, a semigroup possessing an identity element.)
Given a generating set $Y \subseteq S$, and an arbitrary subset $A \subseteq Y$, we use the shorthand $\1_A \in S$ for the sum over $A$:
\begin{equation}
    \label{1 notation}
    \1_{A} \coloneqq \sum_{a \in A} a.
\end{equation}
Given a finite subset $B \subset S$, we write $\N B$ to denote the submonoid of $S$ generated by $B$.

On the other hand, given any finite set $X$, we will write $\N^X$ to denote the free monoid with basis $X$.
We write elements of $\N^X$ as tuples $(n_x)_{x \in X}$, where each $n_x \in \N$.
Note that in this setting, the element $\1_A$ in~\eqref{1 notation} is the usual 0-1 indicator vector of $A \subseteq X$.
We define the \emph{support map} via
\begin{equation}
    \label{supp definition}
    \begin{split}
        \supp : \N^X &\longrightarrow 2^X , \\
        (n_x)_{x \in X} & \longmapsto \{ x : n_x > 0 \}.
    \end{split}
\end{equation}

Back in the general setting, assume that $S$ is $\N$-graded, so that
\begin{equation}
    \label{S grading}
    S = \coprod_{t=0}^\infty S_t, \quad \text{where $S_t + S_{t'} \subseteq S_{t+t'}$ for all $t,t' \in \N$}.
\end{equation}
Upon fixing a field $\k$, the \emph{semigroup ring} of $S$ is the $\k$-algebra with $\k$-basis $\{x^s : s \in S \}$, that is,
\begin{equation}
    \label{k[S]}
    \k[S] \coloneqq \bigoplus_{s \in S} \k x^s,
\end{equation}
where $x$ is a formal indeterminate, with multiplication given by $x^s \cdot x^{s'} = x^{s+s'}$.
The semigroup ring $\k[S]$ inherits a grading from~\eqref{S grading} in the obvious way:
\begin{equation}
    \label{k[S] grading}
    \k[S] = \bigoplus_{t=0}^\infty \k[S]_t, \qquad \text{where } \k[S]_t \coloneqq \bigoplus_{\mathclap{s \in S_t}} \k x^s.
\end{equation}
Note that by~\eqref{k[S] grading},
\begin{equation}
    \label{dim k[S] is size St}
    \dim_\k \k[S]_t = \left| S_t \right|.
\end{equation}
The \emph{Hilbert series} of $\k[S]$ is the generating function of the dimension (as a $\k$-vector space) of each graded piece:
\begin{equation}
    \label{Hilb k[S]}
    \Hilb(\k[S]; z) \coloneqq \sum_{t=0}^\infty \dim_{\k} \k[S]_t \, z^t. 
\end{equation}
Since $\k[S]$ is a finitely generated positively graded $\k$-algebra, its Hilbert series is a rational function.
This leads to the notion of the \emph{Krull dimension}
\begin{equation}
    \label{Kdim k[S]}
    d = \Kdim \k[S] \coloneqq \text{the order of the pole of $\Hilb(\k[S]; z)$ at $z=1$},
\end{equation}
as well as the \emph{multiplicity}
\begin{equation}
    \label{e(k[S])}
    e(\k[S]) \coloneqq \lim_{z \rightarrow 1} (1-z)^{d} \Hilb(\k[S];z).
\end{equation}
By the classical theory of Hilbert functions of graded rings, if the Hilbert function $H(t) \coloneqq \dim_{\k} \k[S]_t$ eventually (for $t \gg 0$) agrees with a quasi-polynomial in $t$, then
\begin{equation}
    \label{Kdim and H(t)}
    \Kdim \k[S] = 1 + \text{degree of $H(t)$}
\end{equation}
and
\begin{equation}
    \label{e and H(t)}
    e( \k[S] ) = ( \text{degree of $H(t)$} )! \cdot ( \text{leading coefficient of $H(t)$} ).
\end{equation}

\subsection*{Ehrhart theory}

Next we specialize the preceding theory to the case where $S$ is the affine semigroup associated to a rational polytope;
standard references are Stanley~\cites{Stanley1980, StanleyCCA}, Miller--Sturmfels~\cite{MillerSturmfels}*{Ch.~7}, Barvinok~\cite{Barvinok}, Beck--Robins~\cite{BeckRobins}, and Bruns--Gubeladze~\cite{BrunsGubeladze}.
A \emph{rational
polytope} $\P \subset \R^n$ is the convex hull of finitely many
points of $\mathbb{Q}^n$.
The \emph{dimension} of~$\P$, denoted by $\dim \P$, is the dimension of its affine hull.
The \emph{intrinsic volume} of $\P$ is its Euclidean volume with respect to its affine hull (rather than with respect to its ambient space $\R^n$). 
The \emph{normalized volume} of $\P$ is given by
\begin{equation}
    \label{Vol}
    \Vol \P \coloneqq (\dim \P)! \cdot (\text{intrinsic volume of $\P$}).
\end{equation}
Note that the standard simplex (of any dimension) has normalized volume $1$.

Embedding $\R^n \hookrightarrow \R^{n+1}$ via $\mathbf{x} \mapsto (\mathbf{x},1)$, we define the rational polyhedral cone
\begin{equation}
    \label{cone P}
    \cone(\P) \coloneqq \big\{ \lambda(\mathbf{x},1) : \mathbf{x} \in \P, \; \lambda \in
    \R_{\ge 0} \big\} \subset \R^{n+1}.
\end{equation}
The \emph{affine semigroup} associated to $\P$ is
\begin{equation}
    \label{S(P)}
    S(\P) \coloneqq \cone(\P) \cap \mathbb{Z}^{n+1},
\end{equation}
which consists of all lattice points in this cone.
Grading $\N^{n+1}$ by the last coordinate gives $S(\P)$ an $\N$-grading by
height $t$, so that~\eqref{S grading} specializes to
\begin{equation}
    \label{S(P) grading}
    S(\P) = \coprod_{t=0}^\infty S(\P)_t, \quad \text{where }
  S(\P)_t \coloneqq (t \P \cap \mathbb{Z}^n) \times \{t\}.
\end{equation}
The number of lattice points in the $t$-fold dilation of $\P$ is denoted by
\begin{equation}
    \label{L(Pt)}
    L(\P;t) \coloneqq \bigl| t \P \cap \mathbb{Z}^n \bigr| = \left|  S(\P)_t \right|.
\end{equation}
In particular, $L(\P;0) = 1$.
By results of Ehrhart~\cite{Ehrhart62} and McMullen~\cite{McMullen78}, 
\begin{equation}
    \label{L quasi-poly}
    \begin{array}{l}
    \text{$L(\P;t)$ is a quasi-polynomial in $t$, whose degree equals $\dim \P$,} \\
    \text{and whose leading coefficient is the intrinsic volume of $\P$}.
    \end{array}
\end{equation}
The associated generating function
\begin{equation}
    \label{Ehr(Pz)}
    \Ehr(\P;z) \coloneqq \sum_{t=0}^\infty L(\P;t)\, z^t,
\end{equation}
in the formal indeterminate $z$,
is the \emph{Ehrhart series} of $\P$.
Note that the Ehrhart series of $\P$ equals the Hilbert series of the semigroup ring $\k[S(\P)]$, since
\begin{equation}
    \label{Ehr Hilb}
    \Ehr(\P;z) = \sum_{t=0}^\infty L(\P;t) \, z^t = \sum_{t=0}^\infty \left| S(\P)_t \right| \, z^t = \sum_{t=0}^\infty \dim_{\k} \k[S(\P)]_t \, z^t = \Hilb(\k[S(\P)]; z),
\end{equation}
where the equalities follow from~\eqref{Ehr(Pz)}, \eqref{L(Pt)}, \eqref{dim k[S] is size St}, and~\eqref{Hilb k[S]}, respectively.
Because of this connection, some authors call $\k[S(\P)]$ the \emph{Ehrhart ring} of $\P$.
The following key lemma highlights the dictionary between the algebraic invariants of the Ehrhart ring, on one hand, and the geometric properties of the polytope itself, on the other hand.

\begin{lemma}
    \label{lemma:Kdim and e are dim and vol}
    Let $\P$ be a rational polytope, with affine semigroup $S(\P)$ as defined in~\eqref{S(P)}, and Ehrhart ring $\k[S(\P)]$ as defined in~\eqref{k[S]}.
    \begin{enumerate}[label=\textup{(\arabic*)}]
        \item \label{Kdim = dim P + 1} $\Kdim \k[S(\P)] = \dim \P + 1$, where $\Kdim$ is the Krull dimension~\eqref{Kdim k[S]}.

        \item \label{e = Vol} $e(\k[S(\P)]) = \Vol \P$, where $e$ is the multiplicity~\eqref{e(k[S])}, and $\Vol$ the normalized volume~\eqref{Vol}.
    \end{enumerate}
\end{lemma}

\begin{proof}
    It is immediate from~\eqref{Ehr Hilb} that $L(\P; t)$ is the Hilbert function $H(t)$ of $\k[S(\P)]$, in the sense of~\eqref{Kdim and H(t)}--\eqref{e and H(t)}.
    Part (1) now follows from~\eqref{Kdim and H(t)} and~\eqref{L quasi-poly}.
    Likewise, combining~\eqref{e and H(t)} and~\eqref{L quasi-poly} with~\eqref{Vol}, we have
    \[
        e( \k[S(\P)]) = (\dim \P)! \cdot (\text{intrinsic volume of $\P$}) \eqcolon \Vol(\P),
    \]
    which proves part (2).
\end{proof}

\subsection*{Stanley decompositions}

As mentioned in the introduction, our technique in this paper involves a type of commutative monoid decomposition due to Stanley~\cite{Stanley82}*{Thm.~5.2} which we have found to be useful in writing down the Krull dimension and (especially) the multiplicity of a semigroup ring $\k[S]$.
Let $S$ be a finitely generated commutative monoid.
Given a finite subset $B \subset S$, let $\N B$ denote the submonoid generated by $B$.
A \emph{Stanley decomposition} of $S$ is a finite disjoint union
\begin{equation}
    \label{Stanley decomp S}
    S = \coprod_i \: (a_i + \N B_i),
\end{equation}
where each $a_i \in S$, and each $B_i \subset S$ is finite and linearly independent (therefore $\N B_i \cong \N^{B_i}$ is a free submonoid of $S$ with basis $B_i$).
A Stanley decomposition~\eqref{Stanley decomp S} of $S$ induces the following vector space decomposition (still called a Stanley decomposition) of the semigroup ring $\k[S]$ from~\eqref{k[S]}:
\begin{equation}
    \label{Stanley decomp of k[S]}
    \k[S] = \bigoplus_i x^{a_i} \cdot \k\big[ \{x^{b} : b \in B_i \} \big].
\end{equation}
Each direct summand in~\eqref{Stanley decomp of k[S]} is called a \emph{Stanley space}.
Writing $\deg s = t$ if $s \in S_t$, observe that given a Stanley decomposition~\eqref{Stanley decomp S} of $S$, we can express the Hilbert series of $\k[S]$ as a finite sum of rational functions:
\begin{equation}
    \label{Hilb k[S] with Stanley}
    \Hilb(\k[S];z) = \sum_i \frac{z^{\deg a_i}}{\prod_{b \in B_i} (1 - z^{\deg b})}.
\end{equation}

\begin{lemma}
    \label{lemma:Stanley decomp gives Kdim and e}
    Let $S$ be a finitely generated commutative monoid admitting a Stanley decomposition
        \[
            S = \coprod_{i} \: (a_i + \N B_i)
        \]
    as in~\eqref{Stanley decomp S}.
    \begin{enumerate}[label=\textup{(\arabic*)}]
    
        \item \label{Kdim = max} The Krull dimension~\eqref{Kdim k[S]} of the semigroup ring $\k[S]$ in~\eqref{k[S]} is given by
        \[
            d \coloneqq \Kdim \k[S] = \max_i \left| B_i \right|.
        \]
    
    \item \label{e = sum over max} The multiplicity~\eqref{e(k[S])} of $\k[S]$ is given by
        \[
            e(\k[S]) = \sum_{\substack{i: \\ \left| B_i \right| = d}} \frac{1}{\prod_{b \in B_i} \deg b}.
        \]

    \end{enumerate}
    
\end{lemma}

\begin{proof}
    Both parts follow from comparing~\eqref{Hilb k[S] with Stanley} with the facts~\eqref{Kdim k[S]} and~\eqref{e(k[S])}, respectively.
\end{proof}

(In the context of Ehrhart theory, see~\cites{BIS16,Stanley1980} for the geometric meaning of a Stanley decomposition of the affine semigroup of a rational polytope.)
Before giving the next series of lemmas, we first recall that given sets $A \subseteq B$, one defines the Boolean interval
    \begin{equation}
        \label{Boolean interval}
        [A,B] \coloneqq \{C : A \subseteq C \subseteq B \}.
    \end{equation}

\begin{lemma}
    \label{lemma:support}
    Let $X$ be a finite set, and let $Z \subseteq 2^X$.
    Recall the map $\supp : \N^X \longrightarrow 2^X$ from~\eqref{supp definition}.
    If $Z$ admits a finite disjoint union of Boolean intervals
    \[
        Z = \coprod_i \: [A_i, B_i],
    \]
    then $\supp^{-1}(Z) \subseteq \N^X$ admits the Stanley decomposition
    \[
        \supp^{-1}(Z) = \coprod_i \: ( \1_{A_i} + \N B_i ).
    \]
\end{lemma}

\begin{proof}
    We have
    \begin{align*}
        \supp^{-1}(Z) &= \supp^{-1} \left( \coprod_i \: [A_i, B_i] \right) \\
        &= \coprod_i \supp^{-1}\left( [A_i, B_i] \right) \\
        &= \coprod_i \Big\{ \mathbf{n} \in \N^X : A_i \subseteq \supp(\mathbf{n}) \subseteq B_i \Big\} \\
        &= \coprod_i \Big\{ (n_x)_{x \in X} : n_x > 0 \text{ for all } x \in A_i, \text{ and } n_x = 0 \text{ for all } x \notin B_i \Big\} \\
        &= \coprod_i \: (\1_{A_i} + \N B_i). \qedhere
    \end{align*}
\end{proof}

We will apply Lemma~\ref{lemma:support} in two specific settings, which we discuss next.
In the first of these settings, the role of $X$ is played by a finite poset.
Let $P$ be a finite poset, and recall that a \emph{lower order ideal} of $P$ is a subset $I \subseteq P$ such that $y \leq x \in I$ implies $y \in I$.
We write
    \begin{equation}
        \label{I(P)}
        \I(P) \coloneqq \{ \text{lower order ideals of $P$} \}.
    \end{equation}
For each $I \in \I(P)$, we write
    \begin{equation}
        \label{max(I)}
        \max(I) \coloneqq \{ \text{maximal elements in $I$} \}.
    \end{equation}
(Viewed as a map, ``$\max$'' is the well-known bijection between $\I(P)$ and the set of antichains in~$P$.)

\begin{lemma}
    \label{lemma:poset partition}
    If $P$ is a finite poset, then
    \[
        2^P = \coprod_{I \in \I(P)} [ \max(I), \: I].
    \]
\end{lemma}

\begin{proof}
    Let $Q \subseteq P$, and let ${\downarrow \!\! Q}$ be the lower closure of $Q$, that is, $\downarrow Q$ is the smallest lower order ideal of $P$ that contains $Q$.
    It suffices to show that $Q \in [\max(I), \: I]$ if and only if $I = {\downarrow \!\! Q}$.
    In one direction, it is clear that $Q \in [\max(\downarrow \! \! Q), \: \downarrow \!\! Q]$, since $Q \subseteq \downarrow \!\! Q$ by definition, and any maximal element of $\downarrow \!\! Q$ must belong to $Q$ (otherwise it would be strictly less than some element of $Q \subseteq \downarrow\!\!Q$, which is a contradiction).
    In the other direction, suppose that $I \in \I(P)$ such that $Q \in [\max(I), \: I]$.
    Since $Q \subseteq I$, we have $\downarrow \!\! Q \subseteq I$.
    But also, since $\max(I) \subseteq Q$, we have $I = \downarrow\!\!\max(I) \subseteq \downarrow\!\! Q$.
    Hence we have both $I \supseteq \downarrow\!\! Q$ and $I \subseteq \downarrow \!\! Q$, and thus $I = \downarrow \!\! Q$.
    This completes the proof.
\end{proof}

The second setting where we will use Lemma~\ref{lemma:support} is that of simplicial complexes; we follow the exposition in Bj\"orner--Wachs~\cite{BjornerWachs}.
An \emph{abstract simplicial complex} $\Delta$ on a vertex set $V$ is a finite collection of subsets of $V$, called \emph{faces}, such that if $B \in \Delta$ and $A \subseteq B$, then $A \in \Delta$.
The maximal faces of~$\Delta$ (with respect to inclusion) are called \emph{facets}, and we write
    \begin{equation}
        \label{F(Delta)}
        \F(\Delta) \coloneqq \{ \text{facets of $\Delta$} \}.
    \end{equation}
We say that $\Delta$ is \emph{shellable} if there exists a linear ordering $F_1, \ldots, F_m$ of the facets in $\F(\Delta)$ such that, for all indices $i < k$, there exists some $j < k$ and a vertex $v \in F_k$ such that
    \begin{equation}
        \label{shelling condition}
        (F_i \cap F_k) \subseteq (F_j \cap F_k) = (F_k \setminus \{v\}).
    \end{equation}
Such an ordering of the facets is called a \emph{shelling order} on $\F(\Delta)$, or just a \emph{shelling} of $\Delta$.
A shelling of $\Delta$ (if one exists) is generally not unique.
Any shelling $F_1, \ldots, F_m$ induces a \emph{restriction map} $\res : \F(\Delta) \longrightarrow \Delta$ given by
    \begin{equation}
        \label{restriction definition}
        \res(F_k) \coloneqq \Big\{ v \in F_k : (F_k \setminus \{v\}) \subseteq F_j \text{ for some $j < k$} \Big\}.
    \end{equation}
The subset $\res(F_k)$ is called the \emph{restriction} of the facet $F_k$.
The following lemma gives a straightforward criterion to check whether an ordering of $\F(\Delta)$, equipped with an explicit map $R$ satisfying~\eqref{restriction definition}, is truly a shelling with restriction map $R$.

\begin{lemma}
    \label{lemma:shelling criterion}
    Let $\Delta$ be an abstract simplicial complex.
    Let $F_1, \ldots, F_m$ be an ordering of the facets of $\Delta$, and let $R: \F(\Delta) \longrightarrow \Delta$ be a map satisfying
    \begin{equation}
        \label{R property}
        R(F_k) = \Big\{ v \in F_k : (F_k \setminus \{v\}) \subseteq F_j \textup{ for some } j < k \Big\} \quad \textup{for all } 1 \leq k \leq m.
    \end{equation}
    If $R(F_k) \not\subseteq F_i$ for all pairs $i<k$, then $F_1, \ldots, F_m$ is a shelling of $\Delta$ with induced restriction map $R$.
\end{lemma}

\begin{proof}
    Let $1 \leq i < k \leq m$.
    Since $R(F_k) \not\subseteq F_i$, there exists some $v \in \big( R(F_k) \setminus F_i \big)$.
    In particular, since $v \in R(F_k)$, by~\eqref{R property} there exists some $j < k$ such that \begin{equation}
        \label{Fk minus x in Fj}
        (F_k \setminus \{v\}) \subseteq F_j.
    \end{equation}
    Combining the fact that $v \notin F_i$ with~\eqref{Fk minus x in Fj}, we have
    \begin{equation}
        \label{two containnments}
        (F_i \cap F_k) \subseteq (F_k \setminus \{v\}) \subseteq (F_j \cap F_k).
    \end{equation}
    Note that if we were to have  $v \in F_j$, then by~\eqref{Fk minus x in Fj} we would have $F_k \subseteq F_j$, which would force $F_k = F_j$ since facets are maximal with respect to inclusion;
    but this would contradict $j < k$, and hence it must be the case that $v \notin F_j$.
    It follows that
    \begin{equation}
        \label{Fj cap Fk in Fk minus x}
        (F_j \cap F_k) \subseteq (F_k \setminus \{v\}).
    \end{equation}
    Comparing~\eqref{Fj cap Fk in Fk minus x} with the opposite inclusion in~\eqref{two containnments} forces the equality $(F_k \setminus \{v\}) = (F_j \cap F_k)$, and thus~\eqref{two containnments} becomes
    \[
        (F_i \cap F_k) \subseteq (F_j \cap F_k) = (F_k \setminus \{v\}).
    \]
    Thus by~\eqref{shelling condition}, since our choice of $i < k$ was arbitrary,  $F_1, \ldots, F_m$ is a shelling of $\Delta$.
    Consequently, this shelling has an induced restriction map $\res$ given by~\eqref{restriction definition}.
    But the defining formula in~\eqref{restriction definition} is identical to~\eqref{R property}, so $\res(F_k) = R(F_k)$ for all $ 1 \leq k \leq m$.
    Hence $R$ is the induced restriction map, as claimed.
\end{proof}

\begin{lemma}[\cite{BjornerWachs}*{Prop.~2.5}]
    \label{lemma:complex partition}

    Let $\Delta$ be an abstract simplicial complex.
    Given a shelling of $\Delta$ with restriction map $\res$, we have
    \[
        \Delta = \coprod_{F \in \F(\Delta)} [ \res(F), \: F].
    \]
\end{lemma}

\begin{lemma}
    \label{lemma:section}
    Let $S$ be a finitely generated commutative monoid.
    Let $X$ be a set, and $\iota : X \longrightarrow S$ an injective map such that $\iota(X)$ generates $S$.
    Let $\pi$ be the canonical extension of $\iota$ given by
    \begin{align*}
        \pi : \N^X & \longrightarrow S, \\
        (n_x)_{x \in X} & \longmapsto \sum_{x \in X} n_x \iota(x),
    \end{align*}
    a surjective homomorphism of monoids.
    Let $\sigma : S \longrightarrow \N^X$ be a section of $\pi$, that is, $\pi \circ \sigma = \operatorname{id}_S$.
    If
    \[
        \sigma(S) = \coprod_{i} \: (a_i + \N B_i)
    \]
    is a Stanley decomposition of $\sigma(S) \subseteq \N^X$, then
    \[
        S = \coprod_{i} \: (\pi(a_i) + \N \pi(B_i))
    \]
    is a Stanley decomposition of $S$.
    
\end{lemma}

\begin{proof}
    Since $\pi \circ \sigma = \operatorname{id}_S$, the map $\pi$ restricts to a bijection $\sigma(S) \longrightarrow S$ with inverse $\sigma$; hence the images under $\pi$ of the Stanley spaces in $\sigma(S)$ still partition $S$.
    Moreover, since $\pi$ is a homomorphism, $\pi(a_i + \N B_i) = \pi(a_i) + \N \pi(B_i)$.
    Finally, the independence of $\pi(B_i)$ in $S$ follows from the fact that $\pi$ is injective on $\sigma(S)$.
\end{proof}

\section{Osculating pairs of partitions and their associated complex}
\label{sec:combinatorics}

In this section, we build up the combinatorial model underlying the Stanley decomposition we will construct in Section~\ref{sec:Stanley decomp}.
The main characters are what we call \emph{osculating pairs of partitions} $(\al,\be)$, and an associated shellable complex $\Delta(\al,\be)$.
In reading this section, the reader may find it helpful occasionally to glance ahead to the comprehensive Example~\ref{ex:big example} on page~\pageref{ex:big example}, where we illustrate our constructions in full detail for a given pair $(\al, \be)$.

The natural habitat for our constructions will be the set $[p] \times [q]$, which we will always depict as a rectangular grid of points with matrix coordinates (i.e., with the point $(i,j)$ in the $i$th row from the top and the $j$th column from the left).
When we refer to \emph{adjacency} and \emph{connectedness} in $[p] \times [q]$, we are viewing $[p] \times [q]$ as the grid graph wherein two points are adjacent if and only if they differ by $\pm( 1,0)$ or $\pm(0, 1)$. 

Recall that an \emph{integer partition} (or just a \emph{partition}) is a (possibly empty) finite sequence $\al = (\al_1, \ldots, \al_\ell)$ of weakly decreasing positive integers.
Equivalently --- and far more usefully in this paper --- one can view a partition $\al = (\al_1, \ldots, \al_\ell)$ as the set
\begin{equation}
    \label{partition as set}
    \al = \Big\{ (i,j) \in (\mathbb{Z}_{>0})^2 : j \leq \al_i \Big\},
\end{equation}
where $\al_i \coloneqq 0$ for all $i > \ell$.
Thus $|\al| = \sum_i \al_i$.
This viewpoint~\eqref{partition as set} is essentially the \emph{Young diagram} of $\al$, which is ubiquitous in algebraic combinatorics.
In this paper, we will depict a partition $\al$ by embedding the set~\eqref{partition as set} into either the northeast or southwest corner of $[p] \times [q]$, as described in part~(\ref{osculating}) of the following definition.

\begin{dfn}
    \label{def:Parpq and Opq}
    Fix positive integers $p$ and $q$.
    \begin{enumerate}
        \item \label{corners} Let
        \[
            \Par_{pq} \coloneqq \Big\{ \text{partitions }\al \subseteq [p-1] \times [q-1]
            \Big\}.
        \]
        Given $\al \in \Par_{pq}$, define the subset
        \begin{align*}
            \cor(\al) &\coloneqq \Big\{ (i,j) \in \al : (i+1,j) \notin \al \text{ and } (i,j+1) \notin \al \Big\} \\
            & = \Big\{ (i,j) \in \al : (\al \setminus \{(i,j)\} ) \in \Par_{pq} \Big\}.
        \end{align*}
        (The notation is due to the fact that such elements are typically called the \emph{corners} of $\al$.)

    \item \label{poset interp} 
        Equivalently, one can view $[p-1] \times [q-1]$ as a poset with respect to the product order, whereby $(i,j) \leq (i', j')$ if and only if $i \leq i'$ and $j \leq j'$.
        Then in the notation from~\eqref{I(P)}--\eqref{max(I)},
        \[
            \Par_{pq} = \mathcal{I}\Big( [p-1] \times [q-1] \Big), \quad \text{and} \quad \cor(\al) = \max(\al).
        \]

    \item \label{frame diagram}
    Define the following reflections of $[p] \times [q]$:
    \begin{align*}
        \NE{i,j} & \coloneqq (i, \: q-j+1), \\
        \SW{i,j} &\coloneqq (p -i+1, \: j).
    \end{align*}
    To each pair $(\al,\be) \in \Par_{pq} \times \Par_{pq}$ we associate the \emph{frame}
    \[
        \Fr(\al,\be) \coloneqq \NE{\al} \cup \SW{\be},
    \]
    which we depict by shading each of its elements with a unit square inside $[p] \times [q]$.
    (The term evokes the way a film director, say,  makes a frame in the air with thumbs and index fingers at two opposite corners of an imaginary screen.)
    
  \item \label{osculating} A pair $(\al, \be) \in \Par_{pq} \times \Par_{pq}$ is said to be \emph{osculating} if $\NE{\al}$ and $\SW{\be}$ are not adjacent to each other.
  (Equivalently, the outlines of $\NE{\al}$ and $\SW{\be}$ in $\Fr(\al,\be)$ form two osculating lattice paths in the usual sense.)
    Let
        \[
            \O_{pq} \coloneqq 
            \Big\{ \text{osculating pairs $(\al, \be) \in \Par_{pq} \times \Par_{pq}$}  \Big\}.
        \]
        
        \item \label{Delta(alpha,beta)} Let $(\al, \be) \in \O_{pq}$.
        For each $j \in [q]$, define the interval
        \[
            N_j \coloneqq N_j(\al, \be) =  \Big\{ i \in [p] : (i,j) \notin \Fr(\al,\be) \Big\}.
        \]
        Then define 
        \[
            \Delta(\al,\be) \coloneqq \Big\{ C \subseteq [p] : N_j \not\subseteq C \textup{ for all } j \in [q] \Big\},
        \]
        that is, the largest abstract simplicial complex on $[p]$ for which the $N_j$'s are nonfaces.
        We will abbreviate its set of facets by
        \[
            \F(\al,\be) \coloneqq \F(\Delta(\al,\be)).
        \]

        \item \label{d definition} Given $(\al, \be) \in \O_{pq}$, define the statistic
        \[
            \d(\al, \be)
            \coloneqq \sum_{\mathclap{\substack{(i,j) \\ \in \cor(\al)}}} ij + \sum_{\mathclap{\substack{(i,j) \\ \in \cor(\be)}}} ij.
        \]
        (The ``$\d$'' is meant to evoke ``degree,'' and is motivated by~\eqref{numerator degree} below.)
    \end{enumerate}

\end{dfn}

As we will show in Lemma~\ref{lemma:F and R}, the following construction gives an explicit shelling $F_1, \ldots, F_m$ of $\Delta(\al,\be)$ along with its restriction map $R$.
The construction begins by picking out the set $\NN(\al,\be)$ of nonfaces of $\Delta(\al,\be)$ which are minimal with respect to inclusion.
We illustrate Construction~\ref{const:F and R} immediately afterward in Example~\ref{ex:big example}.

\begin{construction}
    \label{const:F and R}
    Let $(\al,\be) \in \O_{pq}$.
    First, we define
    \[
        \NN(\al,\be) \coloneqq \{ N_{j_1}, \ldots, N_{j_\ell} \},
    \]
    where the intervals $N_j = N_j(\al,\be)$ are from Definition~\ref{def:Parpq and Opq}(\ref{Delta(alpha,beta)}), and where $j_1, \ldots, j_\ell \in [q]$ are determined recursively as follows.
    An \emph{$\al$-column} is an index $j$ such that $\NE{\cor(\al)}$ contains some $(i,j)$;
    likewise, a \emph{$\be$-column} is an index $j$ such that $\SW{\cor(\be)}$ contains some $(i,j)$, and we also declare $q$ to be a $\be$-column.
    Set $j_1$ equal to the smallest $\be$-column.
    For $i \geq 1$, let $j'_{i+1}$ be the smallest $\al$-column greater than $j_i$, if it exists; then $j_{i+1}$ is the smallest $\be$-column greater than or equal to $j'_{i+1}$.
    Otherwise, the process terminates and $\ell = i$.
    Having thus picked out the indices $j_{1} < \cdots < j_{\ell}$, write
    \begin{equation}
        \label{ai bi}
        N_{j_i} = [a_i, b_i].
    \end{equation}
    Now the following algorithm enumerates certain pairs $(F_k, \:R(F_k))$ such that $R(F_k) \subseteq F_k \subseteq [p]$.
Note that Step~\ref{choice step} is nondeterministic, with each choice of $x_t$ spawning a branch.
Each $F_k$ is built from $[p]$ by removing a sequence of elements $x_1 <  x_2 < \cdots$ (depending on the branch). We first describe the construction of a typical $(F,R(F))$, as a sequence of states $(F^{(t)}, R(F^{(t)}))$ indexed by $t=1,2,3,\ldots$.

\begin{enumerate}[label=\arabic*.,ref=\arabic*]
\setcounter{enumi}{-1}

    \item Initialize $F^{(0)} \coloneqq [p]$ and  $R(F^{(0)})\coloneqq \emptyset$, and set $x_0 \coloneqq 0$ and $b_{i_0} \coloneqq 0$.
    \item\label{recursive step} If $a_i>x_{t-1}$ for some $i$, let $i_t$ be the smallest such index and proceed to Step~\ref{choice step}. 
    Otherwise, output $(F,R(F)) \coloneqq (F^{(t-1)}, \: R(F^{(t-1)}))$ and terminate this branch.
    \item\label{choice step} Choose $x_t \in \left( b_{i_{t-1}}, \:b_{i_t} \right]$.
    \item \label{update step} Set
    \[
        F^{(t)} \coloneqq F^{(t-1)}\setminus\{x_t\},\qquad R(F^{(t)}) \coloneqq R(F^{(t-1)}) \cup \left(x_t, \: b_{i_t} \right],
    \]
    increment $t$, and return to Step~\ref{recursive step}.
\end{enumerate}
Branches are explored depth-first, with priority given to larger values of $x_t$ among the choices available at a given step. 
Enumerating the $(F,R(F))$'s in the order produced by this depth-first search gives a fixed order $(F_1,R(F_1)), \ldots, (F_m, R(F_m))$.
(In fact, the $F_k$'s are ordered lexicographically.)
    
\end{construction}

\begin{example}
    \label{ex:big example}
    Let $p = 10$ and $q = 15$.
    Let
    \[
        \al = (13,9,6,4,2,2,2,1), 
        \qquad
        \be = (12,11,11,8,6,5,3,1).
    \]
    Note that $\al,\be \in \Par_{pq}$, since in each partition the number of parts is less than $p=10$, and the first part is less than $q = 15$.
    To determine whether the pair $(\al,\be)$ belongs to $\O_{pq}$, we depict its corresponding \emph{frame} as in Definition~\ref{def:Parpq and Opq}(\ref{frame diagram}) (we also draw circles to indicate the elements of $\NE{\cor(\al)}$ and $\SW{\cor(\be)}$ for later reference):
    \begin{equation}
        \label{frame in example}
        \begin{tikzpicture}[scale=.3, baseline=(current bounding box.center)]
    
    \fill[lightgray] (15.5,10.5) --++ (0,-8) -- ++ (-1,0) -- ++ (0,1) -- ++(-1,0) -- ++(0,3) -- ++(-2,0) -- ++(0,1) -- ++(-2,0) -- ++(0,1) -- ++(-3,0) -- ++(0,1) -- ++(-4,0) -- ++(0,1) -- cycle;

    \foreach \x/\y in {3/10,7/9,10/8,12/7,14/4,15/3} {
    \node at (\x,\y) [corner] {};
    }

    \fill[lightgray] (0.5,0.5) --++ (0,8) -- ++ (1,0) -- ++ (0,-1) -- ++(2,0) -- ++(0,-1) -- ++(2,0) -- ++(0,-1) -- ++(1,0) -- ++(0,-1) -- ++(2,0) -- ++(0,-1) -- ++(3,0) -- ++(0,-2) -- ++(1,0) -- ++(0,-1) -- cycle;

    \foreach \x/\y in {1/8,3/7,5/6,6/5,8/4,11/3,12/1} {
    \node at (\x,\y) [corner] {};
    }
    
    \foreach \x in {1,...,10}{\foreach \y in {1,...,15}{\node [dot] at (\y,\x) {};}}

    \node [left] at (0,5.5) {$\Fr(\al,\be) \; = $};
    
\end{tikzpicture}
    \end{equation}

    \noindent Since the two shaded regions $\NE{\al}$ and $\SW{\be}$ do not overlap or share any edges, we indeed have $(\al,\be) \in \O_{pq}$.
    To compute the statistic $\d(\al,\be)$ in Definition~\ref{def:Parpq and Opq}(\ref{d definition}), we observe that for each $(i,j) \in \cor(\al)$, the quantity $ij$ is simply the number of points in the rectangular region weakly northeast of $\NE{i,j}$;
    the same is true upon replacing $\al$ by $\be$, and ``northeast''/NE by ``southwest''/SW.
    Hence $\d(\al,\be)$ is the sum of these values $ij$ shown below:
    \[
        \d(\al,\be) = 
        \sum 
        \left(
        \begin{tikzpicture}[scale=.3, baseline=(current bounding box.center)]
    
    \fill[lightgray!30] (15.5,10.5) --++ (0,-8) -- ++ (-1,0) -- ++ (0,1) -- ++(-1,0) -- ++(0,3) -- ++(-2,0) -- ++(0,1) -- ++(-2,0) -- ++(0,1) -- ++(-3,0) -- ++(0,1) -- ++(-4,0) -- ++(0,1) -- cycle;

    \fill[lightgray!30] (0.5,0.5) --++ (0,8) -- ++ (1,0) -- ++ (0,-1) -- ++(2,0) -- ++(0,-1) -- ++(2,0) -- ++(0,-1) -- ++(1,0) -- ++(0,-1) -- ++(2,0) -- ++(0,-1) -- ++(3,0) -- ++(0,-2) -- ++(1,0) -- ++(0,-1) -- cycle;
    
    \foreach \x in {1,...,10}{\foreach \y in {1,...,15}{\node [dot] at (\y,\x) {};}}

    \foreach \x/\y in {3/10,7/9,10/8,12/7,14/4,15/3} {
    \ijBox{\x}{\y}
    }

    \node [scale=.6] at (3,10) {13};
    \node [scale=.6] at (7,9) {18};
    \node [scale=.6] at (10,8) {18};
    \node [scale=.6] at (12,7) {16};
    \node [scale=.6] at (14,4) {14};
    \node [scale=.6] at (15,3) {8};

    \foreach \x/\y in {1/8,3/7,5/6,6/5,8/4,11/3,12/1} {
    \ijBox{\x}{\y}
    }

    \node [scale=.6] at (1,8) {8};
    \node [scale=.6] at (3,7) {21};
    \node [scale=.6] at (5,6) {30};
    \node [scale=.6] at (6,5) {30};
    \node [scale=.6] at (8,4) {32};
    \node [scale=.6] at (11,3) {33};
    \node [scale=.6] at (12,1) {12};

    \node at (1,0.25) {};
    
\end{tikzpicture}
        \right) = 253.
    \]
    Finally, we will obtain a shelling of $\Delta(\al,\be)$ by implementing Construction~\ref{const:F and R}.
    To begin, we inspect the circled points in~\eqref{frame in example} and determine the $j_i$'s as follows.
    The $\al$-columns are 3, 7, 10, 12, 14, and 15;
    the $\be$-columns are 1, 3, 5, 6, 8, 11, 12, and 15 (recall that $q$ is always considered to be a $\be$-column).
    As always, $j_1$ is the smallest $\be$-column, which in this case gives $j_1 = 1$.
    Starting \emph{after} this column, we sweep toward the right until we arrive at the next $\al$-column $j'_2 = 3$;
    then starting \emph{in} this column, we sweep toward the right until we arrive at the next $\be$-column $j_2 = 3$.
    Repeating this process, we obtain the first two rows in the following table, from which the $j_i$'s yield $\NN(\al,\be)$:
    \[
        \begin{array}{c|cccccc}
            i & 1 & 2 & 3 & 4 & 5 & 6 \\ \hline
            j'_i & - & 3 & 7 & 10 & 12 & 14 \\ 
            j_i & 1 & 3 & 8 & 11 & 12 & 15 \\ \hline
            a_i & 1 & 2 & 3 & 4 & 5 & 9 \\
            b_i & 2 & 3 & 6 & 7 & 9 & 10
        \end{array}
        \qquad \leadsto \qquad
        \NN(\al,\be) = \{N_1, N_3, N_8, N_{11}, N_{12}, N_{15} \}.
    \]
    The last two rows follow immediately from~\eqref{ai bi}; for example, at $i=3$, we have $j_3 = 8$, and thus $N_{j_3} = N_8 = [3,6] = [a_3, b_3]$.
    (In words: look at column 8 in~\eqref{frame in example} and find the interval of unshaded row indices, which are rows 3 through 6.
    Thus $N_8 = [3,6]$.)
    Finally, to obtain the sequence of $(F_k, R(F_k))$'s, we implement the depth-first algorithm in Construction~\ref{const:F and R}, which we illustrate below:
    \[
    \ytableausetup{smalltableaux,centertableaux}
    \arraycolsep=1.3pt
        \begin{array}{c|rp{2ex}p{2ex}p{2ex}p{2ex}p{2ex}p{2ex}p{2ex}c|rcp{2.5ex}p{2.5ex}p{2.5ex}p{2.5ex}p{2.5ex}p{2.5ex}p{2.5ex}p{2.5ex}p{2.5ex}p{2.5ex}p{2.5ex}p{2.5ex}p{2.5ex}p{2.5ex}p{2.5ex}p{2.5ex}p{2.5ex}p{2.5ex}}
          \quad & \quad i: & & 1 & 2 & 3 & 4 & 5 & 6 & & \quad k: & & 1 & 2 & 3 & 4 & 5 & 6 & 7 & 8 & 9 & 10 & 11 & 12 & 13 & 14 & 15 & 16 & 17 \\ \hline
           \begin{ytableau}

           \none \\
           \none[\scriptstyle1] \\ 
           \none[\scriptstyle2] \\
           \none[\scriptstyle3] \\
           \none[\scriptstyle4] \\
           \none[\scriptstyle5] \\
           \none[\scriptstyle6] \\
           \none[\scriptstyle7] \\
           \none[\scriptstyle8] \\
           \none[\scriptstyle9] \\
           \none[\scriptstyle10] 
           \end{ytableau}
           & \quad [a_i, b_i]: & & 
            \ydiagram{0,1,1,0,0,0,0,0,0,0,0} & 
            \ydiagram{0,0,1,1,0,0,0,0,0,0,0} & 
            \ydiagram{0,0,0,1,1,1,1,0,0,0,0} & 
            \ydiagram{0,0,0,0,1,1,1,1,0,0,0} & 
            \ydiagram{0,0,0,0,0,1,1,1,1,1,0} & 
            \ydiagram{0,0,0,0,0,0,0,0,0,1,1} & & \begin{array}{r} 
            F_k: \\
            \\
            \;\; \textcolor{gray}{R(F_k):}
            \end{array}
            &  &
            
            \ytableaushort{\none,{},\times,{},{},{},\times,{},{},{},\times} &
            \ytableaushort{\none,{},\times,{},{},{},\times,{},{},\times,{\rshade}} &
            \ytableaushort{\none,{},\times,{},{},\times,{\rshade},{},{},{},\times} &
            \ytableaushort{\none,{},\times,{},{},\times,{\rshade},{},{},\times,{\rshade}} &
            \ytableaushort{\none,{},\times,{},\times,{\rshade},{\rshade},{},{},\times,{}} &
            \ytableaushort{\none,{},\times,{},\times,{\rshade},{\rshade},{},\times,{\rshade},\times} &
            \ytableaushort{\none,{},\times,{},\times,{\rshade},{\rshade},\times,{\rshade},{\rshade},\times} &
            \ytableaushort{\none,{},\times,\times,{\rshade},{\rshade},{\rshade},\times,{},{},\times} &
            \ytableaushort{\none,{},\times,\times,{\rshade},{\rshade},{\rshade},\times,{},\times,{\rshade}} &
            \ytableaushort{\none,\times,{\rshade},\times,{},{},{},\times,{},{},\times} &
            \ytableaushort{\none,\times,{\rshade},\times,{},{},{},\times,{},\times,{\rshade}} &
            \ytableaushort{\none,\times,{\rshade},\times,{},{},\times,{\rshade},{},{},\times} &
            \ytableaushort{\none,\times,{\rshade},\times,{},{},\times,{\rshade},{},\times,{\rshade}} &
            \ytableaushort{\none,\times,{\rshade},\times,{},\times,{\rshade},{\rshade},{},{},\times} &
            \ytableaushort{\none,\times,{\rshade},\times,{},\times,{\rshade},{\rshade},{},\times,{\rshade}} &
            \ytableaushort{\none,\times,{\rshade},\times,\times,{\rshade},{\rshade},{\rshade},{},\times,{}} &
            \ytableaushort{\none,\times,{\rshade},\times,\times,{\rshade},{\rshade},{\rshade},\times,{\rshade},\times} &
        \end{array}
    \]
    
    \noindent Each column represents the subset of $[p]$ consisting of the boxes in that column, where boxes marked ``$\times$'' are omitted.
    In particular, for each $k$, the ``$\times$'' symbols are precisely the $x_t$'s chosen in Step~\ref{choice step}, which are deleted in Step~\ref{update step}.
    The shaded boxes in each $F_k$ represent the subset $R(F_k) \subseteq F_k$;
    in particular, the (possibly empty) strip of consecutive shaded boxes beneath each $x_t$ equals the interval $(x_t, b_{i_t}]$ appended to $R(F^{(t-1)})$ in Step~\ref{update step}.
    This concludes Construction~\ref{const:F and R}.
    Explicitly,
    \begin{align*}
        F_1 &= \{1,3,4,5,7,8,9\}, & R(F_1) &= \emptyset,\\
        F_2 &= \{1,3,4,5,7,8,10\}, & R(F_2) &= \{10\},\\
        F_3 &= \{1,3,4,6,7,8,9\}, & R(F_3) &= \{6\}, \\
        F_4 &= \{1,3,4,6,7,8,10\}, & R(F_4) &= \{6,10\}, \\
        & \vdots & & \vdots \\
        F_{17} &= \{2,5,6,7,9\}, & R(F_{17}) &= \{2,5,6,7,9\}.
    \end{align*}
\end{example}

\begin{lemma}
\label{lemma:F and R}
    
    Let $(\al, \be) \in \O_{pq}$.
    We have
    \[
        \F(\al,\be) = \{F_1, \ldots, F_m \},
    \]
    where the $F_k$'s are given by Construction~\ref{const:F and R}.
    Moreover, $F_1, \ldots, F_m$ is a shelling of $\Delta(\al, \be)$ with restriction map $R$, that is,
    \[
        \res(F_k) = R(F_k).
    \]       
\end{lemma}

\begin{proof}
    First we claim that $\NN(\al\,\be)$ is the set of minimal nonfaces of $\Delta(\al,\be)$, where a nonface $N$ is \emph{minimal} if every proper subset of $N$ belongs to $\Delta(\al,\be)$.
    That is, we claim that
    \begin{equation}
        \label{claim N minimal nonfaces}
        \text{$N_{j_1}, \ldots, N_{j_\ell}$ are all distinct, and are precisely the minimal nonfaces of $\Delta(\al,\be)$}.
    \end{equation}
    In Definition~\ref{def:Parpq and Opq}(\ref{Delta(alpha,beta)}), $\Delta(\al,\be)$ is defined to be the largest abstract simplicial complex on $[p]$ such that the intervals $N_1, \ldots, N_q$ are nonfaces;
    therefore to prove~\eqref{claim N minimal nonfaces}, it suffices to show that for all $j \in [q]$,
    \begin{equation}
        \label{nonminimal nonfaces excluded}
        j \notin \{j_1, \ldots, j_\ell\} \text{ if and only if $N_{j_i} \subseteq N_j$ for some $j_i \neq j$}.
    \end{equation}
    Consider such $j$'s while implementing Construction~\ref{const:F and R}:
    to begin, since $j_1$ is the first $\be$-column, we confirm that for all $j < j_1$, we have $N_{j_1} \subseteq N_j$, because $\min N_{j_1} \leq \min N_{j}$ while $\max N_{j_1} = \max N_j$.
    Similarly, at each stage of the construction:
    for all $j_i < j < j'_{i+1}$, we have $N_{j_i} \subset N_j$, because $\min N_{j_i} = \min N_j$ while $\max N_{j_i} < \max N_j$.
    Likewise, for all $j'_{i+1} \leq j < j_{i+1}$, we have $N_{j+1} \subseteq N_j$ because $\min N_{j+1} \geq \min N_j$ while $\max N_{j+1} = \max N_j$.
    Finally, for all $j > j_\ell$, it must be the case that $j$ is not an $\al$-column, and therefore we have $N_{j_\ell} \subseteq N_j$ because $\min N_{j_\ell} = \min N_j$ while $\max N_{j_\ell} < \max N_j$.
    This establishes that if $j \notin \{j_1, \ldots, j_\ell\}$, then $N_{j_i} \subseteq N_j$ for some $j_i \neq j$.
    Inversely, it is clear from the construction that the start points and endpoints of the $N_{j_i}$'s strictly increase as $i$ increases from $1$ to $\ell$, so none of the $N_{j_i}$'s contains any other.
    We have thus proved~\eqref{nonminimal nonfaces excluded}, which establishes the claim~\eqref{claim N minimal nonfaces}.

    To prove that $\F(\al,\be) = \{F_1, \ldots, F_m\}$, observe that the set of $x_t$'s chosen in each branch in Construction ~\ref{const:F and R} is a minimal transversal of $\NN(\al,\be)$, meaning that the set of $x_t$'s intersects every $N_{j_i}$, but any proper subset of $x_t$'s does not.
    This follows immediately from Steps~\ref{recursive step}--\ref{choice step}.
    Moreover, the construction produces \emph{all} minimal transversals of $\NN(\al,\be)$, one in each branch.
    Since each $F$ is the complement of the set of $x_t$'s (Step~\ref{update step}), the construction produces all sets $F$ which are maximal among those subsets of $[p]$ which are missing an element from each $N_{j_i}$.
    But then by~\eqref{claim N minimal nonfaces}, the $F_k$'s are maximal among subsets of $[p]$ which are missing an element from each minimal nonface of $\Delta(\al,\be)$, meaning that the $F_k$'s are the facets of $\Delta(\al,\be)$.
    Hence $\F(\al,\be) = \{F_1, \ldots, F_m\}$, as claimed.

    Next we claim that 
    \begin{equation}
        \label{claim R acts like restriction}
        R(F_k) = \Big\{ v \in F_k : (F_k \setminus \{v\}) \subseteq F_j \text{ for some } j < k \Big\} \quad \text{for all $1 \leq k \leq m$}.
    \end{equation}
    Letting $X_k$ denote the set of $x_t$'s chosen in the branch of Construction~\ref{const:F and R} which outputs $F_k$, and taking $[p]$ as our universal set, we have $X_k = \overline{F_k}$;
    thus~\eqref{claim R acts like restriction} is equivalent to
    \begin{equation}
        \label{claim equivalent}
        R(F_k) = \Big\{ v \in \overline{X_k} : (X_k \cup \{v\} ) \supseteq X_j \text{ for some } j < k \Big\} \quad \text{for all } 1 \leq k \leq m.
    \end{equation}
    Since the ordering on the $F_k$'s is obtained by giving priority to larger values of $x_t$, each interval $(x_t, b_{i_t}]$  consists precisely of those elements $v \in \overline{X_k}$ which have already been chosen as $x_t$'s earlier in the branch;
    therefore, taken over all $t$, these $v$'s are precisely the elements in $\overline{X_k}$ such that $X_k \cup \{v\}$ contains some earlier $X_j$.
    Since (by Step~\ref{update step}) the set $R(F_k)$ is the union of these intervals $(x_t, b_{i_t}]$, we have proved~\eqref{claim equivalent}, which establishes the equivalent claim~\eqref{claim R acts like restriction}.
    Moreover, the argument above shows that $R(F_k)$ intersects $X_i$ for all $i < k$, and thus
    \begin{equation}
        \label{claim previous facets}
        R(F_k) \not\subseteq F_i \quad \text{for all } i < k.
    \end{equation}
    Since~\eqref{claim R acts like restriction} and~\eqref{claim previous facets} are precisely the hypotheses of Lemma~\ref{lemma:shelling criterion}, that lemma implies that $F_1, \ldots, F_m$ is a shelling of $\Delta(\al,\be)$ with restriction map $R$.
    \end{proof}

    We conclude this section with two remarks connecting our constructions to other work in the literature.

\begin{remark}
    \label{rem:osculating}
     As mentioned above, $\O_{pq}$ can be viewed as the set of all unordered pairs of osculating lattice paths from $(0,0)$ to $(p,q)$, where the identification is made by tracing the outlines of the two shaded regions in $\Fr(\al,\be)$ from northwest to southeast.
    Also, $\O_{pq}$ can be viewed as the set of  skew Young diagrams fitting inside a $p \times q$ rectangle with no empty rows or columns, where the identification is made by taking the complement of $\Fr(\al,\be)$ in $[p] \times [q]$.
    This set (in the special case $p=3$) is the subject of Stanley's OEIS entry \href{https://oeis.org/A362153}{A362153}.
    Finally, we point out that the condition for $(\al,\be)$ to be in $\O_{pq}$ is very similar to (but slightly stronger than) the condition for $\al$ and $\be$ to be \emph{complementary} in the context of Grassmannian Schubert calculus, which is that $\NE{\al} \cap \SW{\be} = \emptyset$.    
\end{remark}

\begin{remark}
    There is a striking coincidence between the set $\NN(\al, \be)$ of minimal nonfaces in Construction~\ref{const:F and R} and the \emph{kelp bed product} (or \emph{extended Demazure product}) introduced in the Monge-related paper~\cite{Erickson25} by the first author;
    this generalized Tiskin's \emph{seaweed product} in~\cite{Tiskin15} from permutations to biwords.
    (In this remark, we require only \emph{partial permutations}.)
    Let $X \coloneqq \NE{\cor(\al) \cup \{(0,q)\}}$ and $Y \coloneqq \SW{\cor(\be) \cup \{(0,q)\}}$, and let $Y^t$ denote the set obtained from $Y$ by transposing each element $(i,j) \mapsto (j,i)$.
    Then it is fairly straightforward to verify that
    \begin{equation}
        \label{kelp miracle}
        X \star Y^t = \{ (a_i-1, b_i+1) \}_{i=1}^\ell \text{ ``$=$'' } \NN(\al, \be),
    \end{equation}
    where the $\star$ product is defined in~\cite{Erickson25}*{Def.~2.3}, and where the second equality (in quotes) requires one to view each ordered pair $(a_i - 1, b_i + 1)$ as the open interval of integers $a_i - 1 < i < b_i + 1$, which is the closed interval $[a_i, b_i]$ from Construction~\ref{const:F and R}.
    (It is a rare occasion when a reader is \emph{encouraged} to conflate these two conflicting meanings of pairs in round brackets.)
    Although the aforementioned papers~\cites{Erickson25,Tiskin15} do relate to tropical multiplication of simple Monge matrices (essentially via the restriction of our map $\pi$ to $\N^{X_{\sA}}$), nonetheless their context is completely disjoint from the present paper, and thus we are somewhat baffled by their ``prediction''~\eqref{kelp miracle} of our $\NN(\al, \be)$.
\end{remark}

\section{A Stanley decomposition of the Monge semigroup}
\label{sec:Stanley decomp}

We begin this section by specializing the general theory from Section~\ref{sec:preliminaries} to the subject of this paper, namely the family of Monge polytopes $\M_{pq}$.
In particular, we study the structure of the affine semigroup $S(\M_{pq})$.
The main result in this section is Theorem~\ref{thm:Stanley decomp}, which gives a Stanley decomposition of $S(\M_{pq})$ indexed by the pairs $(\al,\be) \in \O_{pq}$ and the facet sets $\F(\al,\be)$ which we constructed in Section~\ref{sec:combinatorics}.
This decomposition will immediately yield the Ehrhart series of $\M_{pq}$ in Corollary~\ref{cor:Ehrhart series}.

\subsection*{The Monge polytope and semigroup}

Recall from the introduction the \emph{Monge condition} on a matrix~$M$:
\begin{equation}
    \label{Monge condition in body}
    M_{ij} + M_{IJ} \leq M_{iJ} + M_{Ij} \text{ for all } i < I \text{ and } j < J.
\end{equation}

\begin{dfn}
    \label{def:Mpq}
    Given positive integers $p$ and $q$, the \emph{Monge polytope} is the set
    \[
        \M_{pq} \coloneqq \Big\{ M \in \R^{p \times q}_{\geq 0} : \text{$M$ satisfies~\eqref{Monge condition in body} and $\sum_{i,j} M_{ij} = 1$} \Big\}.
    \]
    By~\eqref{S(P)}--\eqref{S(P) grading}, the \emph{Monge semigroup} is the commutative monoid
    \[
        S(\M_{pq}) 
        = \coprod_{t=0}^\infty \Big\{ M \in \N^{p \times q} : \text{$M$ satisfies~\eqref{Monge condition in body} and $\sum_{i,j} M_{ij} = t$} \Big\} \times \{t\}.
    \]
    \textbf{NB:} In writing elements of $S(\M_{pq})$, we will typically drop the $t$ from $(M,t)$, since $t = \sum_{i,j} M_{ij}$.
    Thus we will write ``Let $M \in S(\M_{pq})$ \ldots'' etc.
    Note that, with respect to the induced $\N$-grading~\eqref{S(P) grading},
    \begin{equation}
        \label{deg M}
        \deg M = \sum_{i,j} M_{ij}.
    \end{equation}
\end{dfn}

To realize $\M_{pq}$ as a polytope, we view it as the subset of the ambient space $\R^{p \times q}$ cut out by the system~\eqref{Monge condition in body} of linear inequalities, and restricted to the simplex $\Delta_{pq-1}$ in which all coordinates are nonnegative and sum to 1.
Since $\Delta_{pq-1}$ has dimension $pq-1$, we likewise have
\begin{equation}
    \label{dim Mpq}
    \dim \M_{pq} = pq - 1.
\end{equation}
Note that $\Vol \M_{pq}$ as defined in~\eqref{Vol} gives the volume of $\M_{pq}$ relative to $\Delta_{pq-1}$.

Next we will establish a minimal generating set for the Monge semigroup $S(\M_{pq})$.
We will show that these generators are parametrized by the (external) disjoint union
\begin{equation}
        \label{X}
        X \coloneqq
        \underbrace{\Big( [p-1] \times [q-1] \Big)}_{\eqcolon \, X_{\sA}} \; \sqcup \; \underbrace{\Big( [p-1] \times [q-1] \Big)}_{\eqcolon \, X_{\sB}} \; \sqcup \underbrace{\vphantom{\Big(}[p]}_{\eqcolon \,  X_{\sC}} \sqcup \underbrace{\vphantom{\Big(}[q]}_{\eqcolon \, X_{\sD}},
    \end{equation}
where the explicit parametrization is given by the following inclusion map $\iota : X \longrightarrow S(\M_{pq})$.

\begin{dfn}
    \label{def:iota ABCD}
    Define the map
    \[
        \iota \coloneqq \sA \sqcup \sB \sqcup \sC \sqcup \sD
    \]
    induced on $X = X_{\sA} \sqcup X_{\sB} \sqcup X_{\sC} \sqcup X_{\sD}$ by the following maps $\sA, \sB, \sC, \sD$, where $J_{a \times b}$ denotes the $a \times b$ matrix with all entries equal to $1$:
    \begin{align*}
        \sA : X_{\sA} &\longrightarrow S(\M_{pq}), &  
        \sC : X_{\sC} & \longrightarrow S(\M_{pq}), \\[1ex]
        \sA(i,j) &= \left[
        \begin{array}{c|c}
        0 & J_{i \times j} \\
        \hline
        0 & 0
        \end{array}
        \right]_{\textstyle .} & 
        \sC(i) &= \left[
        \begin{array}{c}
        0 \\
        \hline
        J_{1 \times q} \\
        \hline
        0
        \end{array}
        \right]_{\textstyle .} {\scriptstyle \textup{$\leftarrow$ row $i$}} \\[5ex]
        \sB : X_{\sB} &\longrightarrow S(\M_{pq}), & 
        \sD: X_{\sD} & \longrightarrow S(\M_{pq}), \\[1ex]
        \sB(i,j) &= \left[
        \begin{array}{c|c}
        0 & 0 \\
        \hline
        J_{i \times j} & 0
        \end{array}
        \right]_{\textstyle .} &
        \sD(j) &= \underset{\substack{\uparrow \\ \textup{column $j$}}}{\left[
        \begin{array}{c|c|c}
        0 & J_{p \times 1} & 0
        \end{array}
        \right]_{\textstyle .}}
    \end{align*}
    By~\eqref{deg M}, with respect to the natural grading~\eqref{S(P) grading} on $S(\M_{pq})$,
    \begin{equation}
        \label{deg ABCD}
        \deg \sA(i,j) = \deg \sB(i,j) = ij, \quad \deg \sC(i) = q, \quad \deg \sD(j) = p.
    \end{equation}
\end{dfn}

\begin{lemma}
    \label{lemma:rays}
    The set
    \[
        \iota(X) = \Big\{ \sA(i,j) \Big\}_{(i,j) \in X_{\sA}} \cup \Big\{ \sB(i,j) \Big\}_{(i,j) \in X_{\sB}} \cup \Big\{ \sC(i) \Big\}_{i \in X_{\sC}} \cup \Big\{ \sD(j) \Big\}_{j \in X_{\sD}}
    \]
    is a minimal generating set for the Monge semigroup $S(\M_{pq})$.
    
    \end{lemma}

    \begin{proof}
        The fact that $\iota(X)$ generates $S(\M_{pq})$ follows from~\eqref{M decomposed} below, which exhibits an arbitrary $M \in S(\M_{pq})$ as an $\N$-linear combination of elements in $\iota(X)$.
        To show minimality, we observe that by Rudolf--Woeginger~\cite{RW95}*{Lemma 2.5}, the elements of $\iota(X)$ are precisely the primitive lattice points on the extreme rays of the pointed cone $\R_{\geq 0} \M_{pq}$.
        In particular, the dictionary between our maps $\sA$--$\sD$ and the notation of~\cite{RW95}*{\S2} is given by
        \begin{align*}
            \sA(i,j) &= R^{(i, q+j-1)}, &
            \sC(i) &= H^{(i)}, \\
            \sB(i,j)  &= L^{(p-i+1, j)}, &
            \sD(j) &= V^{(j)}.
        \end{align*}
        Thus by~\cite{BrunsGubeladze}*{Prop.~1.20}, $\iota(X)$ is a minimal system of generators for the cone $\R_{\geq 0} \M_{pq}$.
        Hence $\iota(X)$ is contained in every generating set of $S(\M_{pq})$, which proves minimality.
    \end{proof}
    
In writing down elements of $\N^X$, we will move freely between the two canonically isomorphic realizations
    \begin{equation}
        \label{two N^X's}
        \N^X \cong \N^{X_{\sA}} \oplus \N^{X_{\sB}} \oplus \N^{X_{\sC}}  \oplus \N^{X_{\sD}}.
    \end{equation}
As in Lemma~\ref{lemma:section}, $\iota$ extends uniquely to a (surjective, by Lemma~\ref{lemma:rays}) homomorphism of monoids
\begin{equation}
        \label{pi N^X}
        \begin{split}
        \pi: \N^X& \longrightarrow S(\M_{pq}), \\
        (n_x)_{x \in X} & \longmapsto \sum_{x \in X} n_x \iota(x),
        \end{split}
    \end{equation}
which can be written more explicitly using the right-hand side of~\eqref{two N^X's}:
    \begin{equation}
        \label{pi}
        \begin{split}
        \pi: \N^{X_{\sA}} \oplus \N^{X_{\sB}} \oplus \N^{X_{\sC}}  \oplus \N^{X_{\sD}} & \longrightarrow S(\M_{pq}), \\
        \Big( (a_{ij}), \: (b_{ij}), \: (c_i), \: (d_j) \Big) & \longmapsto \sum_{i,j} a_{ij} \sA(i,j) + \sum_{i,j} b_{ij} \sB(i,j) + \sum_{i} c_{i} \sC(i) + \sum_{j} d_{j} \sD(j).
        \end{split}
    \end{equation}
In particular, for any subset $Y = Y_{\sA} \sqcup Y_{\sB} \sqcup Y_{\sC} \sqcup Y_{\sD} \subseteq X$, where $Y_{\bullet} \subseteq X_{\bullet}$, we have
    \begin{equation}
        \label{pi(1_Y) and pi(Y)}
        \begin{split}
        \pi(\1_Y) &= \1_{\sA(Y_{\sA}) \: \cup \: \sB(Y_{\sB}) \: \cup \: \sC(Y_{\sC}) \: \cup \: \sD(Y_{\sD})}, \\[.5ex]
        \pi(Y) &= \sA(Y_{\sA})\cup\sB(Y_{\sB})\cup\sC(Y_{\sC})\cup\sD(Y_{\sD}).
        \end{split}
\end{equation}

\subsection*{A Stanley decomposition \texorpdfstring{of $S(\M_{pq})$}{}}

Similarly to~\eqref{two N^X's}, in writing down elements of $2^X$, we will move freely between the two sets
    \begin{equation}
        \label{two 2^X's}
        2^X \cong 2^{X_{\sA}} \times 2^{X_{\sB}} \times 2^{X_{\sC}}  \times 2^{X_{\sD}}
    \end{equation}
via the canonical bijection
    \begin{equation}
        \label{2^X iso}
        \begin{split}
        2^{X_{\sA}} \times 2^{X_{\sB}} \times 2^{X_{\sC}} \times  2^{X_{\sD}} & \longrightarrow 2^{X}, \\
        (Z_{\sA}, Z_{\sB}, Z_{\sC}, Z_{\sD}) & \longmapsto Z_{\sA} \sqcup Z_{\sB} \sqcup Z_{\sC} \sqcup Z_{\sD}.
        \end{split}
    \end{equation}
The bijection~\eqref{2^X iso} sends a product of Boolean intervals 
    \begin{equation}
        \label{product of Booleans}
        [U_{\sA}, V_{\sA}] \times [U_{\sB}, V_{\sB}] \times [U_{\sC}, V_{\sC}] \times [U_{\sD}, V_{\sD}]
    \end{equation}
(as a subset of the left-hand side of~\eqref{2^X iso}) to the single Boolean interval
    \begin{equation}
        \label{single Boolean}
        [U_{\sA} \sqcup U_{\sB} \sqcup U_{\sC} \sqcup U_{\sD}, \: V_{\sA} \sqcup V_{\sB} \sqcup V_{\sC} \sqcup V_{\sD}]
    \end{equation}
(a subset of the right-hand side of~\eqref{2^X iso}).
For the moment, we focus on the first two factors in~\eqref{2^X iso}:
by equipping $X_{\sA} = X_{\sB} = [p-1] \times [q-1]$ with the product order from Definition~\ref{def:Parpq and Opq}(\ref{poset interp}), and applying Lemma~\ref{lemma:poset partition} twice, we have
\begin{equation}
    \label{2^A partition}
    \begin{split}
        2^{X_{\sA}} \times 2^{X_{\sB}} &= \Bigg(\coprod_{\al \in \Par_{pq}} [\cor(\al), \: \al]\Bigg) \times \Bigg(\coprod_{\be \in \Par_{pq}} [\cor(\be), \: \be]\Bigg) \\
        &= \coprod_{\al, \be \in \Par_{pq}} [\cor(\al), \: \al] \times [\cor(\be), \: \be].
    \end{split}
\end{equation}
Hence, given $(A,B) \in 2^{X_{\sA}} \times 2^{X_{\sB}}$, we can unambiguously define
\begin{equation}
    \label{omega(A,B)}
    \omega(A,B) \coloneqq \text{the unique $(\al, \be) \in \Par_{pq} \times \Par_{pq}$ such that $(A,B) \in [\cor(\al), \: \al] \times [\cor(\be), \: \be]$}.
\end{equation}

\begin{dfn}
    \label{def:sigma}
    We define the map
    \begin{align*}
        \sigma : S(\M_{pq}) &\longrightarrow \N^{X_{\sA}} \oplus \N^{X_{\sB}} \oplus \N^{X_{\sC}} \oplus \N^{X_{\sD}}, \\
        M & \longmapsto \Big( (a_{ij}), \: (b_{ij}), \: (c_i), \: (d_j) \Big)
    \end{align*}
    as follows:
    \begin{enumerate}
        \item \label{dj} Set $d_j = \min_i \{ M_{ij} \}$.
        
        Then let $M' \coloneqq M - \sum_{j} d_j \sD(j)$.
        Note that $M' \in S(\M_{pq})$ has a zero in every column.
        
        \item \label{ci} Set $c_i = \min_j \{M'_{ij} \}$.

        Then let $M'' \coloneqq M' - \sum_{i} c_i \sC(i)$.
        Note that $M'' \in S(\M_{pq})$ has a zero in every row and every column.
        Moreover, since Monge matrices are totally monotone~\cite{Burkard}*{Lemma 2.4}, the leftmost zero in each row of $M''$ moves weakly to the right as one proceeds from the top row to the bottom row, and the same is true of the rightmost zero in each row;
        note that all intervening entries are necessarily also zeros.
        Hence the support of $M''$ (i.e., the set of positions $(i,j)$ such that $M''_{ij} \neq 0$) equals $\Fr(\al,\be)$ for some $(\al,\be) \in \O_{pq}$.
        Thus we may write 
        \[
            M'' = M^\al + M^\be,
        \]
        where $M^\al$ (resp., $M^\be)$ is the restriction of $M''$ to the positions in $\NE{\al}$ (resp., $\SW{\be}$), with zeros everywhere else.
        
        \item \label{aij} Set $a_{ij} = 
        M^\al_{i,q-j+1} + M^\al_{i+1,q-j} - M^\al_{i+1,q-j+1} - M^\al_{i,q-j}$.

        \item \label{bij} Set $b_{ij} = 
        M^\be_{p-i+1,j} + M^\be_{p-i,j+1} - M^\be_{p-i,j} - M^\be_{p-i+1,j+1}$.
        
    \end{enumerate}
\end{dfn}

\begin{example}
    \label{ex:sigma}
    We give an example of the map $\sigma$ from Definition~\ref{def:sigma}.
    Let $p = 4$ and $q = 6$, and let
    \[
        M = 
        \begin{bmatrix}
            5 & 0 & 3 & 8 & 15 & 13 \\
            6 & 1 & 1 & 6 & 8 & 6 \\
            13 & 8 & 5 & 8 & 9 & 7 \\
            12 & 6 & 3 & 6 & 7 & 2
        \end{bmatrix}
        \in S(\M_{4,6}).
    \]
    One can quickly verify that $M \in S(\M_{4,6})$ by checking that $M$ satisfies the Monge condition~\eqref{Monge condition in body};
    in fact, it is well known~\cite{Burkard}*{eqn.~(6)} that it suffices to check the Monge condition on every contiguous $2 \times 2$ submatrix.
    To recover the coordinates of $\sigma(M) = ((a_{ij}), (b_{ij}), (c_i), (d_j))$, we follow the four steps of Definition~\ref{def:sigma}:
    \begin{enumerate}
        \item We have $(d_j) = (5,0,1,6,7,2)$, given by the minimum entry in each column of $M$.
        Subtracting $d_j$ from every entry in column $j$, we have
        \[
            M' = \begin{bmatrix}
                0 & 0 & 2 & 2 & 8 & 11 \\
                1 & 1 & 0 & 0 & 1 & 4 \\
                8 & 8 & 4 & 2 & 2 & 5 \\
                7 & 6 & 2 & 0 & 0 & 0
            \end{bmatrix}.
        \]

        \item We have $(c_i) = (0,0,2,0)$, given by the minimum entry in each row of $M'$.
        Subtracting $c_i$ from every entry in row $i$, we have
        \begin{align*}
            M'' &= \begin{bmatrix}
                \gz & \gz & 2 & 2 & 8 & 11 \\
                1 & 1 & \gz & \gz & 1 & 4 \\
                6 & 6 & 2 & \gz & \gz & 3 \\
                7 & 6 & 2 & \gz & \gz & \gz
            \end{bmatrix} \\
            &= \underbrace{\begin{bmatrix}
                \gz & \gz & 2 & 2 & 8 & 11 \\
                \gz & \gz & \gz & \gz & 1 & 4 \\
                \gz & \gz & \gz & \gz & \gz & 3 \\
                \gz & \gz & \gz & \gz & \gz & \gz
            \end{bmatrix}}_{M^\al} + 
            \underbrace{\begin{bmatrix}
                \gz & \gz & \gz & \gz & \gz & \gz \\
                1 & 1 & \gz & \gz & \gz & \gz \\
                6 & 6 & 2 & \gz & \gz & \gz \\
                7 & 6 & 2 & \gz & \gz & \gz
            \end{bmatrix}}_{M^\be},
        \end{align*}
        where $\al = (4,2,1)$ and $\be = (3,3,2)$.

        \item On inspecting $M^\al$, we find that
        \begin{align*}
        a_{1,2} &= 8+0-2-1 = 5, \\
        a_{1,4} &= 2 + 0 - 0 - 0 = 2, \\ a_{2,2} &= 1 + 0 - 0 - 0 = 1, \\
        a_{3,1} &= 3 + 0 - 0 - 0 = 3,
        \end{align*}
        and all other $a_{ij} = 0$.

        \item On inspecting $M^\be$, we find that
        \begin{align*}
            b_{1,1} &= 7 + 6 - 6 - 6 = 1, \\
            b_{2,2} &= 6 + 0 - 1 - 2 = 3, \\
            b_{2,3} &= 2 + 0 - 0 - 0 = 2, \\
            b_{3,2} &= 1 + 0 - 0 - 0 = 1,
        \end{align*}
        and all other $b_{ij} = 0$.
    \end{enumerate}
    It is straightforward to check in this example that $\pi \circ \sigma(M) = M$, that is,
    \begin{align*}
        M &= \phantom{+} 5 \sA(1,2) + 2 \sA(1,4) + \sA(2,2) + 3 \sA(3,1) \\
        & \phantom{=} + \sB(1,1) + 3 \sB(2,2) + 2 \sB(2,3) + \sB(3,2) \\
        & \phantom{=} + 2 \sC(3) \\
        & \phantom{=} + 5 \sD(1) + \sD(3) + 6 \sD(4) + 7 \sD(5) + 2 \sD(6).
    \end{align*}
    In Lemma~\ref{lemma:Z}, we show that this is no coincidence:
    in general, the expansion above can be viewed as the canonical form of $M$ given by the map $\sigma$.
\end{example}

\begin{lemma}
    \label{lemma:Z}
    The map $\sigma$ in Definition~\ref{def:sigma} is a section of the map $\pi$ from~\eqref{pi}, that is, $\pi \circ \sigma = \operatorname{id}_{S(\M_{pq})}$.
    Moreover, setting
    \[
        Z \coloneqq \Big\{ (A,B,C,D) \in 2^{X_{\sA}} \times 2^{X_{\sB}} \times 2^{X_{\sC}} \times 2^{X_{\sD}} : \omega(A,B) \in \O_{pq} \textup{ and } C \in \Delta(\omega(A,B)) \Big\},
    \]
    we have
    \[
        \sigma(S(\M_{pq})) = \supp^{-1}(Z) \subseteq \N^{X_{\sA}} \oplus \N^{X_{\sB}} \oplus \N^{X_{\sC}} \oplus \N^{X_{\sD}}. 
    \]
\end{lemma}

\begin{proof}
    Note that in $M^\al$, the leftmost column and bottom row are all zeros.
    Consequently, in Definition~\ref{def:sigma}(\ref{aij}), the formula for $a_{ij}$ is the formula for the $(i,j)$ entry of the \emph{density matrix} of $M$; see~\cite{Erickson25}*{p.~5}.
    Thus by~\cite{Erickson25}*{Lemma 3.1}, the $a_{ij}$'s are the unique coefficients such that
    \begin{equation}
        \label{M alpha is A's}
        M^\al = \sum_{i,j} a_{ij} \sA(i,j).
    \end{equation}
    Likewise, rotating this picture by 180 degrees, the $b_{ij}$'s are the unique coefficients such that
    \begin{equation}
        \label{M beta is B's}
        M^\be = \sum_{i,j} b_{ij} \sB(i,j).
    \end{equation}
    Hence since $M'' = M^\al + M^\be$, we have
    \[
        M'' = \sum_{i,j} a_{ij} \sA(i,j) + \sum_{i,j} b_{ij} \sB(i,j).
    \]
    Thus we obtain
    \begin{equation}
        \label{M decomposed}
        \begin{split}
        \pi \circ \sigma(M) & = \pi \Big( (a_{ij}), \: (b_{ij}), \: (c_i), \: (d_j) \Big) \\
        & = \underbrace{\sum_{i,j} a_{ij} \sA(i,j) + \sum_{i,j} b_{ij} \sB(i,j)}_{M''} + \underbrace{\sum_{i} c_{i} \sC(i)}_{M' - M''} + \underbrace{\sum_{j} d_{j} \sD(j)}_{M - M'} \\
        &= M,
        \end{split}
    \end{equation}
    and so $\pi \circ \sigma = \operatorname{id}_{S(\M_{pq})}$.
    This proves that $\sigma$ is a section of $\pi$.

    To show that $\sigma(S(\M_{pq})) = \supp^{-1}(Z)$, it suffices to show both that $\sigma(S(\M_{pq})) \subseteq \supp^{-1}(Z)$, that is,
    \begin{equation}
        \label{supp in Z}
        \supp\big(\sigma(S(\M_{pq}))\big) \subseteq Z,
    \end{equation}
    and also that
    \begin{equation}
        \label{sigma circ pi = id}
        \sigma \circ \pi = \operatorname{id}_{\supp^{-1}(Z)}.
    \end{equation}    
    To this end, let $M \in S(\M_{pq})$, and let $\sigma(M) = ( (a_{ij}), (b_{ij}), (c_i), (d_j) )$, so that
    \[
        \supp( \sigma(M)) = \underbrace{\Big\{ (i,j) : a_{ij} > 0 \Big\}}_{A} \sqcup \underbrace{\Big\{ (i,j) : b_{ij} > 0 \Big\}}_{B} \sqcup \underbrace{\Big\{ i : c_i > 0 \Big\}}_{C} \sqcup \underbrace{\Big\{ j : d_j > 0 \Big\}}_{D}.
    \]
    By the definition of $\al = \al(M)$ in Definition~\ref{def:sigma}(\ref{ci}), and by~\eqref{M alpha is A's}, we have $A \in [\cor(\al), \: \al]$;
    by the same argument $B \in [\cor(\be), \: \be]$.
    Thus $\omega(A,B) = (\al, \be)$, and moreover, since we already observed in Definition~\ref{def:sigma}(\ref{ci}) that $(\al, \be) \in \O_{pq}$, we have $\omega(A,B) \in \O_{pq}$, which is the first condition on $Z$ in the statement of the lemma.
    The second is that $C \in \Delta(\omega(A,B))$, or as we can now rewrite, $C \in \Delta(\al, \be)$.
    To show this, we recall that by Definition~\ref{def:sigma}(\ref{dj}), the matrix
    \[
        M' =  \underbrace{\sum_{i,j} a_{ij} \sA(i,j) + \sum_{i,j} b_{ij} \sB(i,j)}_{M''} + \underbrace{\sum_{i} c_{i} \sC(i)}_{M' - M''}
    \]
    has a zero in every column;
    on the other hand, the support of $M'$ equals the union of the supports of $M^\al$ and of $M^\be$ and of $\sum_{i \in C} \sC(i)$, which therefore must not include a full column.
    Hence by Definition~\ref{def:Parpq and Opq}(\ref{Delta(alpha,beta)}), we have $C \in \Delta(\al, \be)$ as desired.
    This completes the proof of~\eqref{supp in Z}.
    
    Finally, to show~\eqref{sigma circ pi = id}, let $\mathbf{n} = ((a_{ij}), (b_{ij}), (c_i), (d_j)) \in \supp^{-1}(Z)$, and write 
    \[ 
        A \coloneqq \supp(a_{ij}), \qquad 
        B \coloneqq \supp(b_{ij}), \qquad 
        C \coloneqq \supp(c_{i}), \qquad
        D \coloneqq \supp(d_{j}). 
    \]
    Now let
    \[
        M = \pi(\mathbf{n}) = \underbrace{\overbrace{\sum_{i,j}^{\phantom{*}} a_{ij} \sA(i,j) + \sum_{i,j} b_{ij} \sB(i,j)}^{(**)} + \sum_{i} c_{i} \sC(i)}_{(*)} + \sum_{j} d_{j} \sD(j).
    \]
    We need to show that $\sigma(M) = \mathbf{n}$, and so we follow the four steps of Definition~\ref{def:sigma}:
        \begin{enumerate}
            \item     Since $\supp(\mathbf{n}) \in Z$, we have $C \in \Delta(\omega(A,B))$;
            thus from the definition of this complex in Definition~\ref{def:Parpq and Opq}(\ref{Delta(alpha,beta)}), it follows that the matrix $(*)$ has a zero in every column.
            Therefore in $M$, the minimum entry in each column $j$ equals $d_j$, so $\sigma(M)$ recovers $(d_j)$ from $\mathbf{n}$.
            Consequently, we have $M' = (*)$.

            \item Since $\supp(\mathbf{n}) \in Z$, we have $\omega(A,B) \in \O_{pq}$, where $\omega$ was defined in~\eqref{omega(A,B)}; thus from the definition of $\O_{pq}$ in Definition~\ref{def:Parpq and Opq}(\ref{osculating}), it follows that the matrix $(**)$ has a zero in every row (and column).
            Therefore in $M' = (*)$, the minimum entry in each row $i$ equals $c_i$, so $\sigma(M)$ recovers $(c_i)$ from $\mathbf{n}$.
            Consequently, we have $M'' = (**)$.

            \item Since $\omega(A,B) \in \O_{pq}$, the decomposition $M'' = (**) = M^\al + M^\be$ can be carried out as in Definition~\ref{def:sigma}(\ref{ci}), where $(\al, \be) = \omega(A,B)$.
            Thus $M^\al = \sum_{i,j} a_{ij} \sA(i,j)$ and $M^\be = \sum_{i,j} b_{ij} \sB(i,j)$.
            From the uniqueness of the $a_{ij}$'s in the density decomposition of $M^\al$~\eqref{M alpha is A's}, it follows that $\sigma(M)$ recovers $(a_{ij})$ in $\mathbf{n}$.

            \item From the uniqueness of the $b_{ij}$'s in the density decomposition of $M^\be$~\eqref{M beta is B's}, it follows that $\sigma(M)$ recovers $(b_{ij})$ in $\mathbf{n}$.
        \end{enumerate}
        Thus we have $\sigma(M) = \mathbf{n}$.
        This establishes~\eqref{sigma circ pi = id}, and completes the proof of Lemma~\ref{lemma:Z}.
\end{proof}

\begin{theorem}
    \label{thm:Stanley decomp}
    The Monge semigroup (Definition~\ref{def:Mpq}) admits the Stanley decomposition
        \[
            S(\M_{pq}) = \coprod_{\substack{(\al, \be) \\ \in \O_{pq}}} \;\; \coprod_{\substack{F \\ \in \F(\al,\be)}} \Bigg( \1_{\sA(\cor(\al)) \: \cup \: \sB(\cor(\be)) \: \cup \: \sC(\res(F))} + \N \Big( \sA(\al) \cup \sB(\be) \cup \sC(F) \cup \sD([q]) \Big) \Bigg),
        \]
     where
     \begin{itemize}
         \item[] $\cor( \; )$ is given in Definition~\ref{def:Parpq and Opq}(\ref{corners});
         \item[] $\O_{pq}$ is given in Definition~\ref{def:Parpq and Opq}(\ref{osculating});
         \item[] $\F(\al,\be)$ and $\res(F)$ are given in Construction~\ref{const:F and R}, via Lemma~\ref{lemma:F and R};
         \item[] $\sA$, $\sB$, $\sC$, and $\sD$ are defined in Definition~\ref{def:iota ABCD}.
     \end{itemize}
\end{theorem}

\begin{proof}

The set $Z \subseteq 2^{X_{\sA}} \times 2^{X_{\sB}} \times 2^{X_{\sC}} \times 2^{X_{\sD}}$ from Lemma~\ref{lemma:Z} equals
    \begin{align*}
        Z &= \coprod_{(\al, \be) \in \O_{pq}} \Big\{ (A,B) \in 2^{X_{\sA}} \times 2^{X_{\sB}} : \omega(A,B) = (\al, \be) \Big\} \times \Delta(\al, \be) \times 2^{[q]} \\
        &= \coprod_{(\al, \be) \in \O_{pq}} [ \cor(\al), \: \al] \times [ \cor(\be), \: \be] \times \Delta(\al, \be) \times 2^{[q]} \\
        &= \coprod_{(\al, \be) \in \O_{pq}} [ \cor(\al), \: \al] \times [ \cor(\be), \: \be] \times \Bigg( \coprod_{F \in \F(\al,\be)} [ \res(F), \: F] \Bigg) \times 2^{[q]} \\
        & = \coprod_{(\al,\be) \in \O_{pq}}\ \coprod_{F \in \F(\al,\be)}
        [\cor(\al), \: \al] \times [\cor(\be), \: \be]\times [\res(F), \: F] \times [\emptyset, \: [q]],
    \end{align*}
where the third equality follows from Lemmas~\ref{lemma:complex partition} and~\ref{lemma:F and R} applied to $\Delta(\al, \be)$, and the fourth equality uses the fact that $2^{[q]} = [\emptyset, \: [q]]$.
Now we shift perspective from $2^{X_{\sA}} \times 2^{X_{\sB}} \times 2^{X_{\sC}} \times 2^{X_{\sD}}$ to $2^X$ via the bijection~\eqref{2^X iso};
thus we consider $Z \subseteq 2^X$, and by~\eqref{product of Booleans}--\eqref{single Boolean} we continue our computation from above as
    \[
        Z =
        \coprod_{(\al,\be)\in\O_{pq}}\ \coprod_{F\in\F(\al,\be)}   \big[\cor(\al)\sqcup\cor(\be)\sqcup\res(F)\sqcup\emptyset, \;\;
        \al\sqcup\be\sqcup F\sqcup[q]\big].
    \]
    Applying Lemma~\ref{lemma:support} to this partition yields a Stanley decomposition of $\operatorname{supp}^{-1}(Z)\subseteq\N^X$:
    \[
        \operatorname{supp}^{-1}(Z) = \coprod_{(\al,\be)\in\O_{pq}}\ \coprod_{F\in\F(\al,\be)}
        \Big(\1_{\cor(\al) \: \sqcup \: \cor(\be) \: \sqcup \: \res(F) \: \sqcup \: \emptyset} + \N\big(\al\sqcup\be\sqcup F\sqcup[q]\big)\Big).
    \]
    By Lemma~\ref{lemma:Z}, $\operatorname{supp}^{-1}(Z)=\sigma(S(\M_{pq}))$, so this is precisely a Stanley decomposition of $\sigma(S(\M_{pq}))\subseteq\N^X$. 
    Applying the Section Lemma~\ref{lemma:section} gives a Stanley decomposition of $S(\M_{pq})$:
    \[
        S(\M_{pq}) = \coprod_{(\al,\be)\in\O_{pq}}\ \coprod_{F\in\F(\al,\be)}
        \Big(\pi\big(\1_{\cor(\al) \: \sqcup \: \cor(\be) \: \sqcup \: \res(F) \: \sqcup \: \emptyset}\big)
        + \N\,\pi\big(\al\sqcup\be\sqcup F\sqcup[q]\big)\Big).
    \]
    The result now follows from~\eqref{pi(1_Y) and pi(Y)}.
\end{proof}

\begin{corollary}[Restatement of Theorem~\ref{thm:Ehrhart in intro}]
    \label{cor:Ehrhart series}
    The Ehrhart series~\eqref{Ehr(Pz)} of the Monge polytope is given by
    \[
        \Ehr(\M_{pq}; z) = 
        \frac{1}{(1-z^p)^q} \cdot 
        \sum_{(\al, \be) \in \O_{pq}} \frac{z^{\d(\al, \be)} }{\prod_{(i,j) \in \al} (1-z^{ij}) \prod_{(i,j) \in \be} (1-z^{ij})} 
        \left( 
        \sum_{F \in \F(\al,\be)} \frac{z^{q \left| \res(F) \right|} }{ (1-z^q)^{|F|}} \right),
    \]
    where, from Definition~\ref{def:Parpq and Opq}(\ref{d definition}),
    \[
        \d(\al, \be)
        \coloneqq \sum_{\mathclap{\substack{(i,j) \\ \in \cor(\al)}}} ij + \sum_{\mathclap{\substack{(i,j) \\ \in \cor(\be)}}} ij.
    \]
\end{corollary}

\begin{proof}
    By~\eqref{Ehr Hilb},
    \[
        \Ehr(\M_{pq}) = \Hilb(\k[S(\M_{pq})];z).
    \]
    In turn, by~\eqref{Hilb k[S] with Stanley}, the Stanley decomposition in Theorem~\ref{thm:Stanley decomp} yields
    \begin{equation}
        \label{Hilb without degrees}
        \Hilb(\k[S(\M_{pq})];z) = \sum_{(\al, \be) \in \O_{pq}} \;\; \sum_{F \in \F(\al,\be)} \frac{z^{\deg \1_{\sA(\cor(\al)) \: \cup \: \sB(\cor(\be)) \: \cup \: \sC(\res(F))}}}{\prod_{M \in \sA(\al) \: \cup \: \sB(\be) \: \cup \: \sC(F) \: \cup \: \sD([q])} (1-z^{\deg M})}.
    \end{equation}
    Since degree in $S(\M_{pq})$ is additive, and by~\eqref{deg ABCD},
    \begin{align}
        \deg \1_{\sA(\cor(\al)) \: \cup \: \sB(\cor(\be)) \: \cup \: \sC(\res(F))} &= \sum_{(i,j) \in \cor(\al)} \underbrace{\deg \sA(i,j)}_{ij} + \sum_{(i,j) \in \cor(\be)} \underbrace{\deg \sB(i,j)}_{ij} + \sum_{i \in \res(F)} \underbrace{\deg \sC(i)}_{q} \nonumber \\
        & = \d(\al, \be) + q \left| \res(F) \right|. \label{numerator degree}
    \end{align}
In the same way,
    \begin{align}
        & \phantom{{} = {}}  \prod\nolimits_{M \in \sA(\al) \: \cup \: \sB(\be) \: \cup \: \sC(F) \: \cup \: \sD([q])} \: (1-z^{\deg M}) \nonumber \\[1ex]
        &= \prod_{\mathclap{(i,j) \in \al}} (1-z^{\deg \sA(i,j)}) \cdot \prod_{\mathclap{(i,j) \in \be}} (1-z^{\deg \sB(i,j)}) \cdot \prod_{i \in F} (1-z^{\deg \sC(i)}) \cdot \prod_{j \in [q]} (1-z^{\deg \sD(j)}) \nonumber \\
        & = \prod_{\mathclap{(i,j) \in \al}} (1-z^{ij}) \cdot \prod_{\mathclap{(i,j) \in \be}} (1-z^{ij}) \cdot \prod_{i \in F} (1-z^{q}) \cdot \prod_{j \in [q]} (1-z^{p}) \nonumber \\
        &= \prod_{\mathclap{(i,j) \in \al}} (1-z^{ij}) \cdot \prod_{\mathclap{(i,j) \in \be}} (1-z^{ij}) \cdot (1-z^{q})^{|F|} \cdot (1-z^{p})^q. \label{denominator degree}
    \end{align}
The result follows, upon substituting~\eqref{numerator degree} and~\eqref{denominator degree} in the numerator and denominator (respectively) of~\eqref{Hilb without degrees}, and factoring out the factors that are independent of $F$.
\end{proof}

See Appendix~\ref{appendix:3 by 3} for the complete term-by-term details in Corollary~\ref{cor:Ehrhart series} in the special case $p=q=3$.

\section{Volume of the Monge polytope}
\label{sec:main results}

This final section establishes the volume formula previewed in Theorem~\ref{thm:volume in intro}.
Our formula is a sum over certain lattice paths known as \emph{Delannoy paths};
for details about the history and applications of these paths, see the expository article~\cite{BanderierSchwer} and the references therein.
Given a partition $\al = (\al_1, \ldots, \al_\ell)$, we recall the standard shorthand
\[
    \al! \coloneqq \al_1! \cdots \al_\ell!.
\]
We also recall that $\al'$ denotes the \emph{conjugate partition} of $\al$, defined via
\[
    \al'_j = \left| \{ i : \al_i \geq j \} \right|,
\]
or more intuitively from the viewpoint~\eqref{partition as set}, $\al'$ is obtained from $\al$ via $(i,j) \mapsto (j,i)$.

\begin{dfn}\
    \label{def:Delannoy}

    \begin{enumerate}

    \item A \emph{Delannoy  path} in $[p] \times [q]$ is a path from $(1,1)$ to $(p,q)$ formed by steps 
    \[
    \downarrow \, \coloneqq (1,0), \qquad \rightarrow \: \coloneqq (0,1), \qquad \searrow \: \coloneqq (1,1).
    \]
    Although strictly speaking a Delannoy path $D$ is a subset of $[p] \times [q]$, we will sometimes refer to $D$ as if it were a sequence of steps. 
    Let
    \[
        \D_{pq} \coloneqq \Big\{ \text{Delannoy paths in $[p] \times [q]$} \Big\}.
    \]
    
    \item \label{DZ definition} A Delannoy path is said to be \emph{Z-avoiding} if it does not contain a subpath of the form
    \[
        \rightarrow \underbrace{\downarrow \cdots \downarrow}_{\mathclap{\substack{\text{nonempty run} \\ \text{of $\downarrow$'s}}}} \rightarrow,
    \]
    since such a subpath resembles a ``Z'' shape.
    Let
    \[
    \DZ_{pq} \coloneqq 
    \Big\{ D \in \D_{pq} : 
    \text{$D$ is Z-avoiding}
    \Big\}.
    \]
    
    \item \label{dd(D)} 
    Given $D \in \D_{pq}$, let
    \[
        \dd(D) \coloneqq \text{the number of $\searrow$'s in $D$}.
    \]
    
    \item \label{Conn(D)} Given $D \in \D_{pq}$, let
    \[
        \Conn(D) \coloneqq \Big\{ \text{connected components $C \subseteq D$} \Big\},
    \]
    where the term \emph{connected} refers to the grid graph on $[p] \times [q]$.
    That is, two points in $D$ are adjacent if and only if they differ by $\pm( 1,0)$ or $\pm(0, 1)$.
    Consequently,
    \[
        \left| \Conn(D) \right| = \dd(D) + 1.
    \]

    \item \label{vv(D)} Let $D \in \D_{pq}$.
    We say that $V \in \Conn(D)$ is a \emph{vertical component} if $V = [a,b] \times \{j\}$.
    (Equivalently, $V$ is connected by a sequence of $\downarrow$'s which is not adjacent to a $\rightarrow$ on either side.)
    Let
    \[
        \vv(D) \coloneqq \prod_{\substack{\text{vert.~comps.} \\ V \in \Conn(D)} } |V|.
    \]

    \item \label{ff(D)} Given $D \in \D_{pq}$, note that its complement $\overline{D}$ in $[p] \times [q]$  equals $\Fr(\al,\be)$ for some $(\al,\be) \in \O_{pq}$.
    Let
    \[
        \ff(D) \coloneqq \al! \al'! \be! \be'!.
    \]
    \end{enumerate}
\end{dfn}

\begin{theorem}[Restatement of Theorem~\ref{thm:volume in intro}]
    \label{thm:volume in body}
    The normalized volume~\eqref{Vol} of the Monge polytope is given by
    \[
    \Vol(\M_{pq}) = \frac{1}{p^q q^{p-1}} \sum_{D \in \DZ_{pq}} q^{\dd(D)} \: \frac{\vv(D)}{\ff(D)},
    \]
    with all notation as in Definition~\ref{def:Delannoy}.
\end{theorem}

Before giving the proof, we provide a brief example below of a single term corresponding to a path $D \in \DZ_{pq}$.
We will then prove a series of three lemmas leading up to the proof of Theorem~\ref{thm:volume in body}.

\begin{example}
    \label{ex:Delannoy}
    As in Example~\ref{ex:big example}, let $p=10$ and $q=15$.
    A typical path $D \in \DZ_{pq}$ is shown below on the left-hand side;
    on the right-hand side we depict its complement, which is the frame corresponding to a unique pair $(\al,\be) \in \O_{pq}$:

    
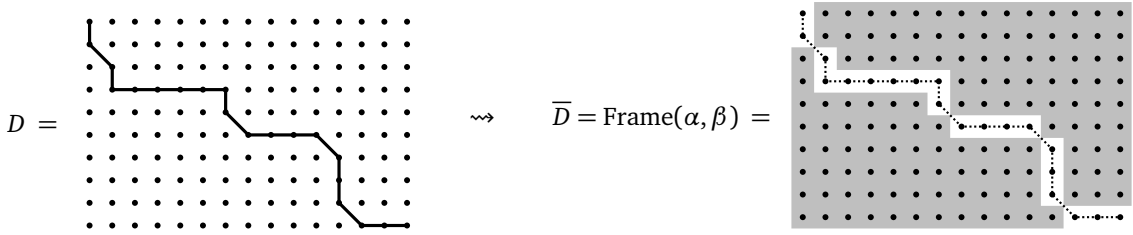
\[
\begin{tikzpicture}[scale=.3]
    
    \foreach \x in {1,...,10}{\foreach \y in {1,...,15}{\node [dot] at (\y,\x) {};}}

    \draw [very thick]
    (1,10) -- ++(0,-1) -- ++(1,-1) --++ (0,-1) -- ++(5,0) -- ++(0,-1) -- ++(1,-1) -- ++(3,0) -- ++(1,-1) -- ++(0,-2) -- ++(1,-1) -- ++(2,0) ;

    \node [left] at (0,5.5) {$D \; = $};
    
\end{tikzpicture}
\qquad
\begin{tikzpicture}[scale=.3]
    
    \fill[lightgray]
    (15.5,10.5) -- ++(0,-9) -- ++(-3,0) -- ++(0,3) -- ++(-1,0) -- ++(0,1) -- ++(-4,0) -- ++(0,2) -- ++(-5,0) -- ++(0,1) -- ++(-1,0) -- ++(0,2) -- cycle;
    
    \fill[lightgray]
    (0.5,0.5) -- ++(0,8) -- ++(1,0) -- ++(0,-2) -- ++(5,0) -- ++(0,-1) -- ++(1,0) -- ++(0,-1) -- ++(4,0) -- ++(0,-3) -- ++(1,0) -- ++(0,-1) -- cycle;
    
    \foreach \x in {1,...,10}{\foreach \y in {1,...,15}{\node [dot] at (\y,\x) {};}}

    \draw [thick, densely dotted]
    (1,10) -- ++(0,-1) -- ++(1,-1) --++ (0,-1) -- ++(5,0) -- ++(0,-1) -- ++(1,-1) -- ++(3,0) -- ++(1,-1) -- ++(0,-2) -- ++(1,-1) -- ++(2,0) ;

    \node [left] at (0,5.5) {$\leadsto \qquad \overline{D} = \Fr(\al,\be) \; = $};
    
\end{tikzpicture}
\]

    \noindent We observe that $D$ contains four $\searrow$'s, and hence
    \begin{equation}
        \label{diag in example}
        \dd(D) = 4.
    \end{equation}
    Consequently, $\left|\Conn(D) \right| = 4+1=5$; the five connected components of $D$ are precisely the components separated by the $\searrow$'s.
    (These are easier to see in the picture of $\overline{D}$ above, where the five unshaded regions are the five connected components of $D$.)
    Of these five connected components, exactly two of them are vertical components, namely those in columns 1 and 12.
    Since the vertical components have cardinality 2 and 3, respectively,
    \begin{equation}
        \label{vert in example}
        \vv(D) = 2 \cdot 3 = 6.
    \end{equation}
    From $\overline{D}$ we read off
    \begin{align*}
        \al &= (14,14,13,8,8,4,3,3,3), & \be &= (12,11,11,11,7,6,1,1), \\
        \al' &= (9,9,9,6,5,5,5,5,3,3,3,3,3,2), & \be' &= (8,6,6,6,6,6,5,4,4,4,4,1),
    \end{align*}
    and hence
    \begin{equation}
        \label{fact in example}
        \ff(D) = \underbrace{(14!)^2 13! (8!)^2 4! (3!)^3}_{\al!} \underbrace{(9!)^3 6! (5!)^4 (3!)^5 2!}_{\al'!} \underbrace{12! (11!)^3 7! 6! (1!)^2}_{\be!} \underbrace{8! (6!)^5 5! (4!)^4 1!}_{\be' !}.
    \end{equation}
    Thus, combining~\eqref{diag in example}--\eqref{fact in example}, the path $D$ contributes the following quantity to the sum in Theorem~\ref{thm:volume in body}:
    \begin{align*}
        q^{\dd(D)} \frac{\vv(D)}{\ff(D)} &= 15^4 \cdot \frac{6}{\ff(D)} \\
        & = 1 \Big/ \left(  \begin{array}{l}
        50026022032502780182578701547632172929813264067209498571 \\
        793990608255340797927573447631273361238307318530048 \times 10^{29}
        \end{array}
        \right).
    \end{align*}
    (The astronomically large denominator is the price paid for having chosen $p$ and $q$ large enough to give an interesting path $D$ for this example.
    We emphasize that this example is meant purely to familiarize the reader with Definition~\ref{def:Delannoy}; in practical computation of Theorem~\ref{thm:volume in body}, as described in the introduction, one can use dynamic programming to avoid performing these calculations on individual paths $D \in \DZ_{pq}$.)
\end{example}

In stating and proving our remaining lemmas, it will be convenient to work with the complement of $\Fr(\al,\be)$;
we call this the \emph{picture} corresponding to $(\al,\be)$, which we denote by
\begin{equation}
    \label{Pic definition}
    \Pic(\al,\be) \coloneqq ([p] \times [q]) \setminus \Fr(\al,\be).
\end{equation}
Given $F \in \F(\al,\be)$, we will denote its complement by $\overline{F} \coloneqq [p] \setminus F$.

\begin{lemma}
    \label{lemma:vol as nice sum}
    We have
    \[
        \Vol(\M_{pq}) = \frac{1}{p^q q^p} \sum_{(\al,\be) \in \O_{pq}} \Bigg(
        \sum_{F \in \widetilde{\F}(\al,\be)}
        \frac{q^{|\overline{F}|}}{\al!\al'!\be!\be'!}
        \Bigg),
    \]
    where
    \[
        \widetilde{\F}(\al,\be) \coloneqq \Big\{ F \in \F(\al,\be) : |\overline{F}| = p+q - \left| \Pic(\al,\be) \right| \Big\}.
    \]
\end{lemma}

\begin{proof}
    Recall from~\eqref{dim Mpq} that $\dim \M_{pq} = pq-1$.
    Thus by Lemma~\ref{lemma:Kdim and e are dim and vol}\ref{Kdim = dim P + 1},
    \begin{equation}
        \label{Kdim = pq}
        \Kdim \k[S(\M_{pq})] = pq.
    \end{equation}
    By Lemma~\ref{lemma:Kdim and e are dim and vol}\ref{e = Vol},
    \begin{equation}
        \label{Vol = e in proof}
        \Vol(\M_{pq}) = e( \k[S(\M_{pq})]).
    \end{equation}
    In turn, applying Lemma~\ref{lemma:Stanley decomp gives Kdim and e}\ref{e = sum over max} to the Stanley decomposition in Theorem~\ref{thm:Stanley decomp} yields
    \begin{equation}
        \label{e in special case}
        e( \k[S(\M_{pq})]) = \sum_{(\al,\be) \in \O_{pq}} \Bigg( \sum_{F} \frac{1}{\prod_{(i,j) \in \al} \deg \sA(i,j) \prod_{(i,j) \in \be} \deg \sB(i,j) \prod_{i \in F} \deg \sC(i) \prod_{j \in [q]} \deg \sD(j)} \Bigg),
    \end{equation}
    where, using~\eqref{Kdim = pq}, the inner sum in~\eqref{e in special case} ranges over the set
    \begin{equation}
        \label{biggest F's}
        \Big\{ F \in \F(\al,\be) : |\al| + |\be| + |F| + q = pq \Big\}.
    \end{equation}
    Rearranging the defining condition in~\eqref{biggest F's}, we have
    \begin{align*}
        |F| & =  
        \underbrace{pq - (|\al| + |\be|)}_{\left| \Pic(\al,\be) \right|} {} - q \\
        p - |F| &= p - \Big( \left|\Pic(\al,\be)\right| - q \Big) \\
       |\overline{F}| &= p+q - \left|\Pic(\al,\be) \right|,
    \end{align*}
    which is the defining condition of the set $\widetilde{\F}(\al,\be)$ in the statement of the lemma.
    Hence the inner sum in~\eqref{e in special case} ranges over all $F \in \widetilde{\F}(\al,\be)$.
    Finally, using~\eqref{deg ABCD} to supply the degrees in~\eqref{Vol = e in proof}--\eqref{e in special case}, we have
        \begin{align*}
        \Vol(\M_{pq}) &= e( \k[S(\M_{pq})]) \\
        &= \sum_{(\al,\be) \in \O_{pq}} \Bigg( \sum_{F \in \widetilde{\F}(\al,\be)} \frac{1}{\prod_{(i,j) \in \al} ij \prod_{(i,j) \in \be} ij \prod_{i \in F} q \prod_{j \in [q]} p} \Bigg) \\
        &= \sum_{(\al,\be) \in \O_{pq}} \Bigg( \sum_{F \in \widetilde{\F}(\al,\be)} \frac{1}{ \underbrace{\textstyle \prod_{i}\left( \prod_{j=1}^{\al_i} j \right) \cdot \prod_{j}\left( \prod_{i=1}^{\al'_j} i \right)}_{ \prod_{(i,j) \in \al} ij} \cdot
        \underbrace{\textstyle \prod_{i}\left( \prod_{j=1}^{\be_i} j \right) \cdot \prod_{j}\left( \prod_{i=1}^{\be'_j} i \right)}_{ \prod_{(i,j) \in \be} ij} {} \cdot q^{|F|} \cdot p^q} \Bigg) \\
        &= \frac{1}{p^q}  \sum_{(\al,\be) \in \O_{pq}} \Bigg( \sum_{F \in \widetilde{\F}(\al,\be)} \frac{1}{\al! \al'! \be! \be'!} \cdot \frac{q^p}{q^{|F|}} \cdot \frac{1}{q^p} \Bigg) \\
        & = \frac{1}{p^q q^p}  \sum_{(\al,\be) \in \O_{pq}} \Bigg( \sum_{F \in \widetilde{\F}(\al,\be)} \frac{q^{|\overline{F}|}}{\al! \al'! \be! \be'!} \Bigg). \qedhere
    \end{align*}
\end{proof}

Next we aim at obtaining an explicit description of the sets $\widetilde{\F}(\al,\be)$.
It will turn out (Lemma~\ref{lemma:F tilde}) that the nonempty $\widetilde{\F}(\al,\be)$'s are in bijective correspondence with the paths in $\DZ_{pq}$.

\begin{lemma}
    \label{lemma:F tilde}
    Let $(\al,\be) \in \O_{pq}$, and let $\widetilde{\F}(\al,\be)$ be the set defined in Lemma~\ref{lemma:vol as nice sum}.
    \begin{enumerate}[label=\textup{(\arabic*)},ref=\arabic*]
        \item \label{only DZ} $\widetilde{\F}(\al,\be) \neq \emptyset$ if and only if $\Pic(\al,\be) \in \DZ_{pq}$.
        
        \item \label{F tilde part 2} If $\Pic(\al,\be) = D \in \DZ_{pq}$, then
        \begin{enumerate}[label=\textup{(\alph*)},ref=\alph*]

        \item \label{size Fbar = diag + 1} $|\overline{F}|
            = \dd(D) + 1$, for all $F \in \widetilde{\F}(\al,\be)$.

            \item \label{size F tilde = vert} $\left| \widetilde{\F}(\al,\be) \right| = \vv(D)$.
            
        \end{enumerate}
    \end{enumerate}
\end{lemma}

\begin{proof}

We will use the following observations, which are immediate from Construction~\ref{const:F and R}:
        \begin{itemize}
            \item If two minimal nonfaces $N, N' \in \NN(\al,\be)$ arise from distinct connected components of $\Pic(\al,\be)$, then $N \cap N' = \emptyset$.
            \item Each connected component $C \in \Conn(\Pic(\al,\be))$ contributes to $\NN(\al,\be)$ precisely those $N_j$'s which are minimal (with respect to inclusion, and ignoring duplicates) among those arising from $C$.

            \item If some minimal nonface $N \in \NN(\al,\be)$ is disjoint from all the others, then it contributes exactly one $x_t$ to every branch of Construction~\ref{const:F and R}, and therefore contributes exactly 1 (as a summand) to $|\overline{F}|$ for every $F \in \F(\al,\be)$.
            Moreover, since the choice of the $x_t \in N$ is always independent of the other choices in Step~\ref{choice step}, we have that $|N|$ divides $|\F(\al,\be)|$.
            In this sense, we will say that ``$N$ contributes $|N|$ as a factor of $| \widetilde{\F}(\al,\be) |$.''
        \end{itemize}

    \begin{table}[t]
        \centering
        \begin{tabular}{|c|c|c|c|}
        \hline
        $C \in \Conn(D)$ 
            & 
        \begin{tabular}{c} Contribution \\
        as element(s) of \\
        $\NN(\al,\be)$ 
        \end{tabular}
            &
        \begin{tabular}{c} Contribution \\
        as a summand of \\ $|\overline{F}|$ \end{tabular} 
            & 
        \begin{tabular}{c}
        Contribution \\
        as a factor of \\
        $\left| \widetilde{\F}(\al,\be) \right|$
        \end{tabular}
            \\[3ex] \hline
        \begin{tikzpicture}
        [scale=.2,baseline=(current bounding box.center)]
            \draw (0,0) node [dot] {} -- ++(0,-1) node [dot] {} -- ++(0,-1) node [dot] {} -- ++(0,-1) node [dot] {} ;
            \node [left] at (0,0) {$(a,j) \rightarrow$} ;
            \node [left] at (0,-3) {$(b,j) \rightarrow$} ;
            \node at (0,1) {};
            \node at (0,-4) {};
        \end{tikzpicture}
            &
        $[a,b]$ 
            & 
        1 
            &
        $|C| = b-a+1$
        \\ \hline
        \begin{tikzpicture}
        [scale=.2,baseline=(current bounding box.center)]
            \draw (0,7) node [dot] {} -- ++(0,-1) node [dot] {} -- ++(0,-1) node [dot] {} -- ++(0,-1) node [dot] {} -- ++(0,-1) node [dot] {} -- ++(1,0) node [dot] {} -- ++(1,0) node [dot] {} -- ++(1,0) node [dot] {} -- ++(1,0) node [dot] {} -- ++(1,0) node [dot] {} -- ++(0,-1) node [dot] {} -- ++(0,-1) node [dot] {} --++(0,-1) node [dot] {} --++(0,-1) node [dot] {} ;
            \node [left] at (0,3) 
            {$(a,j) \rightarrow$} ;
            \node at (0,8) {} ;
            \node at (0,-2) {} ;
            \draw [decoration={brace},decorate] (4,2.5) -- node[below=2pt] {$\scriptstyle{\geq 1}$} (1,2.5);
            \draw [decoration={brace},decorate] (.5,7) -- node[right=2pt] {$\scriptstyle{\geq 0}$} (.5,4);
            \draw [decoration={brace},decorate] (5.5,2) -- node[right=2pt] {$\scriptstyle{\geq 0}$} (5.5,-1);
        \end{tikzpicture} 
            &
        $\{a\}$ 
            & 
        1 
            &
        1
        \\ \hline
        \begin{tikzpicture}
        [scale=.2,baseline=(current bounding box.center)]
            \draw (0,6) node [dot] {} -- ++(0,-1) node [dot] {} -- ++(0,-1) node [dot] {} -- ++(0,-1) node [dot] {} -- ++(1,0) node [dot] {} -- ++(0,-1) node [dot] {} -- ++(0,-1) node [dot] {} --++(0,-1) node [dot] {} ;
            \node [left] at (0,6) {$(a,j) \rightarrow$} ;
            \node [left] at (0,3) {$(b,j) \rightarrow$} ;
            \node [left] at (1,0) {$(c,j+1) \rightarrow$} ;
            \node at (0,7) {};
            \node at (0,-1) {};
        \end{tikzpicture} 
            &
        $[a,b]$ and $[b,c]$ 
            & 
        $
        \begin{cases}
           1 & \text{if } x_t = b, \\
           2 & \text{otherwise}
        \end{cases}
        $
            &
        1
        \\ \hline
        \begin{tabular}{c}
        Any $C$ containing\\
        the following ``Z'' pattern:\\
        \begin{tikzpicture}
        [scale=.2,baseline=(current bounding box.center)]
            \draw (-1,3) node [dot] {} -- ++(1,0) node [dot] {} -- ++(0,-1) node [dot] {} -- ++(0,-1) node [dot] {} -- ++(0,-1) node [dot] {} -- ++(0,-1) node [dot] {} -- ++(0,-1) node [dot] {} -- ++(0,-1) node [dot] {} -- ++(1,0) node [dot] {};
            \node [left] at (-1,3) {$(a,j) \rightarrow$} ;
            \node [right] at (1,-3) {$\leftarrow (b,j+2)$} ;
            \draw [decoration={brace},decorate] (.5,2) -- node[right=2pt] {$\scriptstyle{\geq 0}$} (.5,-2);
        \end{tikzpicture} 
        \end{tabular}
            &
        \begin{tabular}{c}
            at least \\
            $\{a\}$ and $\{b\}$
        \end{tabular}
            & 
        $\geq 2$
            &
        $\widetilde{\F}(\al,\be) = \emptyset$
        \\
        \hline
    \end{tabular}
        \caption{Assuming that $\Pic(\al,\be) = D \in \D_{pq}$, we list all possible types of connected components $C \in \Conn(D)$.
        The second column gives the minimal nonface(s) $N \in \NN(\al,\be)$ arising from each~$C$.
        The third column counts the number of $x_t$'s which are chosen from these $N$'s in Step~\ref{choice step} of Construction~\ref{const:F and R};
        each $x_t$ contributes 1 to $|\overline{F}|$.
        The fourth column counts how many choices of $x_t$'s are possible from these $N$'s when constructing some $F \in \widetilde{\F}(\al,\be)$.}
        \label{table:Conn(D)}
    \end{table}

    To begin, we assume that $\Pic(\al,\be) \in \D_{pq}$.
    (At the end, we will use the following analysis to handle the case $\Pic(\al,\be) \notin \D_{pq}$.)
    Hence, suppose $\Pic(\al,\be) = D \in \D_{pq}$.
    Since a Delannoy path with no $\searrow$'s has cardinality $p+q-1$, and since every additional $\searrow$ replaces two points by one, we have
    \begin{align}
        |D| &= p + q - 1 - \dd(D) \nonumber \\
        &= p + q - \left| \Conn(D) \right|, \label{D = p+q-Conn}
    \end{align}
    where the second equality follows from the fact 
    \begin{equation}
        \label{Conn = diag + 1}
        \left| \Conn(D) \right| = \dd(D) + 1
    \end{equation}
    observed in Definition~\ref{def:Delannoy}(\ref{Conn(D)}).
    Therefore the set $\widetilde{\F}(\al,\be)$ in Lemma~\ref{lemma:vol as nice sum} can be redefined as
    \begin{equation}
        \label{F tilde new}
        \widetilde{\F}(\al,\be) = \Big\{ F \in \F(\al,\be) : |\overline{F}| = \left| \Conn(D) \right|, \text{ where } D = \Pic(\al,\be) \Big\}.
    \end{equation}
    Hence in order to describe the elements $F \in \widetilde{\F}(\al,\be)$, it suffices to consider the contribution to $|\overline{F}|$ coming from each connected component in $D = \Pic(\al,\be)$, which we do in Table~\ref{table:Conn(D)}.
    Observing this table, note that any $C \in \Conn(D)$ contributes at least 1 to $|\overline{F}|$ for every $F \in \F(\al,\be)$.
    Therefore if any $C \in \Conn(D)$ contributes more than~1, it is impossible to satisfy $|\overline{F}| = \left| \Conn(D) \right|$, as is required by the condition on $\widetilde{\F}(\al,\be)$ on~\eqref{F tilde new}.
    Therefore if $\widetilde{\F}(\al,\be) \neq \emptyset$, then $\Pic(\al,\be)$ must not contain the pattern in the last row of Table~\ref{table:Conn(D)}, namely, the ``Z'' pattern in Definition~\ref{def:Delannoy}(\ref{DZ definition}).
    Conversely, if $\Pic(\al,\be)$ is Z-avoiding, then by the first three rows of Table~\ref{table:Conn(D)}, all of its connected components contribute exactly 1 to at least some $|\overline{F}|$, and thus $\widetilde{\F}(\al,\be) \neq \emptyset$ by~\eqref{F tilde new}.
    Therefore, to summarize this paragraph,
    \begin{equation}
        \label{assuming D, must be DZ}
        \text{if $\Pic(\al,\be) \in \D_{pq}$, then $\widetilde{\F}(\al,\be) \neq \emptyset \Longleftrightarrow \Pic(\al,\be) \in \DZ_{pq}$}.
    \end{equation}

    To complete the proof of part~(\ref{only DZ}), we must show that $\Pic(\al,\be) \notin \D_{pq}$ implies $\widetilde{\F}(\al,\be) = \emptyset$.
    Recall from its original definition in Lemma~\ref{lemma:vol as nice sum} that $\widetilde{\F}(\al,\be)$ consists of those facets $F \in \F(\al,\be)$ such that 
    \begin{equation}
        \label{F tilde original}
        |\overline{F}| = p + q - \left| \Pic(\al,\be) \right|.
    \end{equation}
    In particular, as we have just seen,~\eqref{F tilde original} is satisfied by each $F \in \widetilde{\F}(\al,\be)$ if $\Pic(\al,\be) \in \DZ_{pq}$.
    By contrast, since $\Pic(\al,\be)$ is a skew Young diagram with no empty rows or columns (see Remark~\ref{rem:osculating}), if it is not a Delannoy path then it contains a $2 \times 2$ square.
    We consider the effect of on~\eqref{F tilde original} obtained by augmenting the $Z$-avoiding $C$'s in Table~\ref{table:Conn(D)} by a single point, so as to contain a $2 \times 2$ square (which can be done only with the second or third rows).
    In either case, adding an extra point to complete a $2 \times 2$ square necessarily increases $\left| \Pic(\al,\be) \right|$ by 1;
    but crucially, it does \emph{not} decrease the contribution to $|\overline{F}|$.
    Hence the equation~\eqref{F tilde original} no longer holds for any $F \in \F(\al, \be)$, and therefore $\widetilde{\F}(\al, \be) = \emptyset$.
    To summarize this paragraph,
    \begin{equation}
        \label{if not D then F tilde empty}
        \text{if $\Pic(\al,\be) \notin \D_{pq}$, then $\widetilde{\F}(\al,\be) = \emptyset$}.
    \end{equation}
    Part~\eqref{only DZ} of the lemma now follows from~\eqref{assuming D, must be DZ} and~\eqref{if not D then F tilde empty}.

    Part~(\ref{F tilde part 2}\ref{size Fbar = diag + 1}) of the lemma follows immediately from~\eqref{Conn = diag + 1} and~\eqref{F tilde new}.
    Part~(\ref{F tilde part 2}\ref{size F tilde = vert}) follows from the rightmost column in Table~\ref{table:Conn(D)}, since the top row is a vertical component, and since the next two rows contribute factors of~1.
\end{proof}

\begin{proof}[Proof of Theorem~\ref{thm:volume in body}]

    Starting from the volume formula given in Lemma~\ref{lemma:vol as nice sum}, we have
    \begin{align*}
        \Vol(\M_{pq}) &= \frac{1}{p^q q^p} \sum_{(\al,\be) \in \O_{pq}} \Bigg(
        \sum_{F \in \widetilde{\F}(\al,\be)}
        \frac{q^{|\overline{F}|}}{\al!\al'!\be!\be'!}
        \Bigg) & \text{by Lemma~\ref{lemma:vol as nice sum}}\\
        &= \frac{1}{p^q q^p} \sum_{\substack{(\al,\be) \in \O_{pq}: \\ \Pic(\al,\be) \in \DZ_{pq}}} \Bigg(
        \sum_{F \in \widetilde{\F}(\al,\be)}
        \frac{q^{|\overline{F}|}}{\al!\al'!\be!\be'!}
        \Bigg) & \text{by Lemma~\ref{lemma:F tilde}(\ref{only DZ})} \\
        &= \frac{1}{p^q q^p} \sum_{\substack{(\al,\be) \in \O_{pq}: \\ \Pic(\al,\be) \in \DZ_{pq}}} \Bigg(
        \sum_{F \in \widetilde{\F}(\al,\be)}
        \frac{q^{|\overline{F}|}}{\ff(\Pic(\al,\be))}
        \Bigg) & \text{by Definition~\ref{def:Delannoy}(\ref{ff(D)})} \\
        &= \frac{1}{p^q q^p} \sum_{\substack{(\al,\be) \in \O_{pq}: \\ \Pic(\al,\be) \in \DZ_{pq}}} \Bigg(
        \sum_{F \in \widetilde{\F}(\al,\be)}
        \frac{q^{\dd(\Pic(\al,\be))+1}}{\ff(\Pic(\al,\be))}
        \Bigg) & \text{by Lemma~\ref{lemma:F tilde}(\ref{F tilde part 2}\ref{size Fbar = diag + 1})} \\
        &= \frac{1}{p^q q^p} \sum_{\substack{(\al,\be) \in \O_{pq}: \\ \Pic(\al,\be) \in \DZ_{pq}}} \Bigg(
        \vv(\Pic(\al,\be)) \cdot     \frac{q^{\dd(\Pic(\al,\be))+1}}{\ff(\Pic(\al,\be))}
        \Bigg) & \text{by Lemma~\ref{lemma:F tilde}(\ref{F tilde part 2}\ref{size F tilde = vert})} \\
        &= \frac{1}{p^q q^p} \sum_{D \in \DZ_{pq}} q^{\dd(D)+1} \cdot \frac{\vv(D)}{\ff(D)} & \text{putting $D = \Pic(\al,\be)$} \\
        &= \frac{1}{p^q q^{p-1}} \sum_{D \in \DZ_{pq}} q^{\dd(D)} \: \frac{\vv(D)}{\ff(D)},
    \end{align*}
    where the substitution $D = \Pic(\al,\be)$ in the penultimate line is justified by the fact that the ``$\Pic$'' map~\eqref{Pic definition} is injective, and its image contains all of $\D_{pq}$ (and therefore $\DZ_{pq}$).
\end{proof}

As a coda to the paper, having introduced the notion of Z-avoiding Delannoy paths, it was natural to ask for the precise number of such paths for arbitrary $p$ and $q$.
The following proposition answers this question via a generating function.

\begin{proposition}
    \label{prop:count Dpq}
    Let $\DZ_{pq}$ be the set defined in Definition~\ref{def:Delannoy}(\ref{DZ definition}).
    We have
    \[
        \sum_{p,q=1}^\infty \left| \DZ_{pq} \right| x^p y^q = \frac{xy(1 -x + xy)}{1- 2x - y + x^2 + xy -x^2y^2}.
    \]
\end{proposition}

\begin{proof}
    Denote the desired generating function by 
    \[
        D^{\mathrm{Z} \!\!\!\! \backslash} \coloneqq \sum_{p,q=1}^\infty \left| \DZ_{pq} \right| x^p y^p.
    \]
    Recalling that $\D_{pq}$ denotes the set of \emph{all} Delannoy paths from $(1,1)$ to $(p,q)$, it is clear that
    \begin{equation}
        \label{D equation}
        D \coloneqq \sum_{p,q=1}^\infty \left| \D_{pq} \right| x^p y^q = \frac{xy}{1-(x+y+xy)},
    \end{equation}
    since each path step $\downarrow$, $\rightarrow$, or $\searrow$ contributes $x$, $y$, or $xy$, respectively, to the corresponding term in~$D$, and we are considering all possible words in these three steps. 
    (The $xy$ in the numerator of $D$ is required because our paths start at $(1,1)$ rather than at $(0,0)$.)
    Now let $\mathcal{Z}_{pq} \subset \D_{pq}$ denote the subset of paths that contain a forbidden ``Z'' pattern
    \begin{equation}
        \label{Z pattern}
        \rightarrow \underbrace{\downarrow \cdots \downarrow}_{\mathclap{\substack{\text{nonempty run} \\ \text{of $\downarrow$'s}}}} \rightarrow
    \end{equation}
    from Definition~\ref{def:Delannoy}(\ref{DZ definition}), and set $Z \coloneqq \sum_{p,q=1}^\infty \left| \mathcal{Z}_{pq} \right| x^p y^q$.
    We then have
    \begin{equation}
        \label{DZ = D - Z}
        D^{\mathrm{Z} \!\!\!\! \backslash} = D - Z.
    \end{equation}
    Each path in $\mathcal{Z}_{pq}$ must contain a leftmost ``Z'' pattern~\eqref{Z pattern}, and thus we can decompose it uniquely into the following four consecutive subpaths (in order):
    \begin{equation}
        \label{four segments}
        \begin{array}{l}
        \text{an initial Z-avoiding segment path ending with $\rightarrow$}; \\
        \text{a nonempty run of $\downarrow$'s}; \\
        \text{a $\rightarrow$}; \\
        \text{a final arbitrary path}.
        \end{array}
    \end{equation}
    To translate~\eqref{four segments} into generating functions, let $\mathcal{E}_{pq} \subset \DZ_{pq}$ denote the subset of paths ending in~$\rightarrow$, and set $E \coloneqq \sum_{p,q=1}^\infty \left| \mathcal{E}_{pq} \right| x^p y^q$.
    Then~\eqref{four segments} is equivalent to the fourfold factorization
    \begin{equation}
        \label{Z equation}
        Z = E \cdot \frac{x}{1-x} \cdot y \cdot \frac{D}{xy} = \frac{DE}{1-x},
    \end{equation}
    where we divided the $D$ by $xy$ because the final arbitrary path should not include the shift by $(1,1)$ from~\eqref{D equation}.
    Now, each path in $\mathcal{E}_{pq}$ is obtained by appending a $\rightarrow$ to the end of a Z-avoiding path which does not end in the pattern $\rightarrow \downarrow \cdots \downarrow$ (where the run of $\downarrow$'s is nonempty).
    Thus, let $\mathcal{C}_{pq} \subset \DZ_{pq}$ denote the subset of paths that \emph{do} end in the pattern $\rightarrow \downarrow \cdots \downarrow$, and put $C \coloneqq \sum_{p,q=1}^\infty \left| \mathcal{C}_{pq} \right| x^p y^q$.
    We have then said that
    \begin{equation}
        \label{E = (DZ-C)y}
        E = (D^{\mathrm{Z} \!\!\!\! \backslash}-C)y.
    \end{equation}
    Moreover, each path in $\mathcal{C}_{pq}$ is obtained by appending a nonempty run of $\downarrow$'s to the end of some Z-avoiding path ending with $\rightarrow$, that is, to the end of some path of the type counted by $E$.
    Therefore,
    \begin{equation}
        \label{C equation}
        C = E \cdot \frac{x}{1-x},
    \end{equation}
    and by substituting~\eqref{C equation} back into~\eqref{E = (DZ-C)y}, we can express $E$ in terms of $D^{\mathrm{Z} \!\!\!\! \backslash}$ as follows:
    \begin{align}
        E &= \left(D^{\mathrm{Z} \!\!\!\! \backslash} - E \cdot \frac{x}{1-x} \right)y \nonumber \\
        E + E \frac{xy}{1-x} & = D^{\mathrm{Z} \!\!\!\! \backslash}y \nonumber \\
        E \left(\frac{1-x+xy}{1-x} \right) &= D^{\mathrm{Z} \!\!\!\! \backslash}y \nonumber \\
        E &= D^{\mathrm{Z} \!\!\!\! \backslash} \frac{y(1-x)}{1 - x + xy}. \label{E in terms of DZ}
    \end{align}
    Putting this all together, we have
    \begin{align*}
        D^{\mathrm{Z} \!\!\!\! \backslash} & = D - Z & \text{by~\eqref{DZ = D - Z}} \\
        & = D - \frac{DE}{1-x} & \text{by~\eqref{Z equation}} \\
        & = D - DD^{\mathrm{Z} \!\!\!\! \backslash} \frac{y(1-x)}{(1-x)(1-x+xy)} & \text{by~\eqref{E in terms of DZ}} \\
        & = D - DD^{\mathrm{Z} \!\!\!\! \backslash} \frac{y}{1 - x + xy},
    \end{align*}
    and then solving for $D^{\mathrm{Z} \!\!\!\! \backslash}$ we obtain
    \begin{align*}
        D^{\mathrm{Z} \!\!\!\! \backslash} + DD^{\mathrm{Z} \!\!\!\! \backslash} \frac{y}{1-x+xy} &= D \\
        D^{\mathrm{Z} \!\!\!\! \backslash} \left(\frac{1 - x + xy + Dy}{1-x+xy} \right) & = D \\
        D^{\mathrm{Z} \!\!\!\! \backslash} &= \frac{D(1-x+xy)}{1 - x + xy + Dy},
    \end{align*}
    and finally substituting $D = \frac{xy}{1 -x -y -xy}$ from~\eqref{D equation}, we conclude that
    \[
        D^{\mathrm{Z} \!\!\!\! \backslash} = \frac{\frac{xy(1 - x + xy)}{1-x-y-xy}}{1 - x + xy + \frac{xy(y)}{1-x-y-xy}} = \frac{xy(1 - x + xy)}{(1-x+xy)(1-x-y-xy) + xy^2} = \frac{xy(1 - x + xy)}{1 - 2x - y + x^2 + xy - x^2 y^2},
    \]
    which is the rational function given in Proposition~\ref{prop:count Dpq}.
\end{proof}

Proposition~\ref{prop:count Dpq} yields the following concrete values of $\bigl| \DZ_{pq} \bigr|$ for $p,q \leq 10$:
\begin{equation}
    \label{table Dpq}
    \begin{array}{|>{\columncolor{lightgray}}c |c|c|c|c|c|c|c|c|c|c|}
    \hline
    \rowcolor{lightgray} p \backslash q & 1 & 2 & 3 & 4 & 5 & 6 & 7 & 8 & 9 & 10 \\ \hline
    1 & 1 & 1 & 1 & 1 & 1 & 1 & 1 & 1 & 1 & 1 \\ \hline
    2 & 1 & 3 & 4 & 5 & 6 & 7 & 8 & 9 & 10 & 11 \\ \hline
    3 & 1 & 5 & 10 & 16 & 23 & 31 & 40 & 50 & 61 & 73 \\ \hline
    4 & 1 & 7 & 19 & 39 & 67 & 104 & 151 & 209 & 279 & 362 \\ \hline
    5 & 1 & 9 & 31 & 79 & 161 & 287 & 468 & 716 & 1044 & 1466 \\ \hline
    6 & 1 & 11 & 46 & 141 & 336 & 684 & 1249 & 2108 & 3352 & 5087 \\ \hline
    7 & 1 & 13 & 64 & 230 & 631 & 1455 & 2962 & 5500 & 9520 & 15592 \\ \hline
    8 & 1 & 15 & 85 & 351 & 1093 & 2829 & 6385 & 12999 & 24436 & 43121 \\ \hline
    9 & 1 & 17 & 109 & 509 & 1777 & 5117 & 12727 & 28295 & 57610 & 109324 \\ \hline
    10 & 1 & 19 & 136 & 709 & 2746 & 8725 & 23770 & 57463 & 126337 & 257253 \\ \hline
    \end{array}
\end{equation}
Unlike the classical Delannoy numbers, the numbers $\bigl| \DZ_{pq} \bigr|$ in~\eqref{table Dpq} are not symmetric in $p$ and $q$, since the forbidden ``Z'' pattern~\eqref{Z pattern} is not symmetric in $\rightarrow$ and $\downarrow$.

\appendix

\section{The Ehrhart series \texorpdfstring{of $\M_{3,3}$}{for p=q=3}}
\label{appendix:3 by 3}

To illustrate Theorem~\ref{thm:Ehrhart in intro} (equivalently, Corollary~\ref{cor:Ehrhart series}), we write down every term in the Ehrhart series of $\M_{pq}$ in the case $p=q=3$.
The data are organized in the table below.
The 29 rows correspond to the 29 elements of $\O_{3,3}$;
in particular, the first column depicts $\Fr(\al,\be)$ for each $(\al,\be) \in \O_{3,3}$.
In the third column, for each facet $F \in \F(\al,\be)$, we use boldface to indicate the elements (if any) of $\res(F) \subseteq F$.
We use ditto marks (") to indicate that an entry is the same as the one directly above it.

\newcolumntype{M}{>{$}c<{$}} 

\begin{longtable}{M | M | M | M}
        (\al,\be) \in \O_{pq} & \vphantom{\Bigg|} \frac{z^{\d(\al, \be)} }{\prod_{(i,j) \in \al} (1-z^{ij}) \prod_{(i,j) \in \be} (1-z^{ij})} & F \in \F(\al,\be) & \frac{z^{q \left| \res(F) \right|} }{ (1-z^q)^{|F|}}  \\ \hline
        \exIntro{} & 
        1 & 
            \begin{array}{c}
                \{1,2\} \\
                \{1,\mathbf{3}\} \\
                \{\mathbf{2}, \mathbf{3}\}
            \end{array} &
            \begin{array}{c}
                1/(1-z^3)^2 \\
                z^3/(1-z^3)^2 \\
                z^6/(1-z^3)^2
            \end{array}
        \\ \hline
        \exIntro{3/3} & 
        \vphantom{\bigg|} 
        \frac{z}{1-z}
        &
        \begin{array}{c}
                \{1,2\} \\
                \{1,\mathbf{3}\}
            \end{array}
        &
            \begin{array}{c}
                1/(1-z^3)^2 \\
                z^3/(1-z^3)^2
            \end{array} 
        \\ \hline
        \exIntro{3/3,2/3} & 
        \vphantom{\bigg|}
        \frac{z^2}{(1-z)(1-z^2)}
        & \text{"}
        & \text{"}
        \\ \hline
        \exIntro{1/1} & 
        \frac{z}{1-z} &
            \begin{array}{c}
                \{1,3\} \\
                \{\mathbf{2},3\}
            \end{array}
        & \text{"}
        \\ \hline
        \exIntro{1/1,2/1} & 
        \vphantom{\bigg|}
        \frac{z^2}{(1-z)(1-z^2)}  
        & \text{"}
        & \text{"}
        \\ \hline
        \exIntro{1/1,3/3} & 
        \frac{z^2}{(1-z)^2}
        &
        \begin{array}{c}
                \{1,3\} \\
                \{\mathbf{2}\}
            \end{array}
        &
            \begin{array}{c}
                1/(1-z^3)^2 \\
                z^3/(1-z^3)
            \end{array} 
        \\ \hline
        \exIntro{1/1,2/1,3/3} & 
        \vphantom{\bigg|}
        \frac{z^3}{(1-z)^2(1-z^2)}
        & \text{"}
        & \text{"}
        \\ \hline
        \exIntro{3/3,2/3,1/1} & 
        \vphantom{\bigg|}
        \frac{z^3}{(1-z)^2(1-z^2)}
        & \text{"}
        & \text{"} 
        \\ \hline
        \exIntro{1/1,3/3,3/2} & 
        \frac{z^3}{(1-z)^2(1-z^2)}
        &
        \begin{array}{c}
                \{1\} \\
                \{\mathbf{2}\}
            \end{array}
        &
            \begin{array}{c}
                1/(1-z^3) \\
                z^3/(1-z^3)
            \end{array} 
        \\ \hline
        \exIntro{1/1,2/1,3/3,3/2} & 
        \vphantom{\bigg|}
        \frac{z^4}{(1-z)^2(1-z^2)^2}
        & \text{"}
        & \text{"}
        \\ \hline
         \exIntro{1/1,3/3,2/3,3/2} & 
        \vphantom{\bigg|}
        \frac{z^5}{(1-z)^2(1-z^2)^2}
        & \text{"}
        & \text{"}
        \\ \hline
        \exIntro{1/1,3/3,3/2,2/3,2/2} & 
        \vphantom{\bigg|}
        \frac{z^5}{(1-z)^2(1-z^2)^2(1-z^4)}
        & \text{"}
        & \text{"}
        \\ \hline
        \exIntro{1/1,1/2,3/3} & 
        \frac{z^3}{(1-z)^2(1-z^2)}
        &
        \begin{array}{c}
                \{2\} \\
                \{\mathbf{3}\}
            \end{array}
        & \text{"}
        \\ \hline
        \exIntro{1/1,1/2,2/1,3/3} & 
        \vphantom{\bigg|}
        \frac{z^5}{(1-z)^2(1-z^2)^2}
        & \text{"}
        & \text{"}
        \\ \hline
        \exIntro{1/1,1/2,2/1,2/2,3/3} & 
        \vphantom{\bigg|}
        \frac{z^5}{(1-z)^2(1-z^2)^2(1-z^4)}
        & \text{"}
        & \text{"}
        \\ \hline
        \exIntro{1/1,1/2,2/3,3/3} & 
        \vphantom{\bigg|}
        \frac{z^4}{(1-z)^2(1-z^2)^2}
        & \text{"}
        & \text{"} 
        \\ \hline
        \exIntro{3/3,3/2} & 
        \vphantom{\bigg|}
        \frac{z^2}{(1-z)(1-z^2)}
        &
        \begin{array}{c}
                \{1,2\}
            \end{array}
        &
            \begin{array}{c}
                1/(1-z^3)^2
            \end{array} 
        \\ \hline
        \exIntro{3/3,2/3,3/2} & 
        \vphantom{\bigg|}
        \frac{z^4}{(1-z)(1-z^2)^2}
        & \text{"}
        & \text{"}
        \\ \hline
        \exIntro{3/3,3/2,2/3,2/2} & 
        \vphantom{\bigg|}
        \frac{z^4}{(1-z)(1-z^2)^2(1-z^4)}
        & \text{"}
        & \text{"} 
        \\ \hline
        \exIntro{1/1,2/1,2/3,3/3} & 
        \vphantom{\bigg|}
        \frac{z^4}{(1-z)^2(1-z^2)^2}
        &
        \begin{array}{c}
                \{1,3\}
            \end{array}
        & \text{"}
        \\ \hline
        \exIntro{1/1,1/2} & 
        \vphantom{\bigg|} 
        \frac{z^2}{(1-z)(1-z^2)} & \begin{array}{c}
                \{2,3\}
            \end{array}
        & \text{"}
        \\ \hline
        \exIntro{1/1,1/2,2/1} & 
        \vphantom{\bigg|} 
        \frac{z^4}{(1-z)(1-z^2)^2} 
        & \text{"}
        & \text{"}
        \\ \hline
        \exIntro{1/1,1/2,2/1,2/2} & 
        \vphantom{\bigg|} 
        \frac{z^4}{(1-z)(1-z^2)^2(1-z^4)}
        & \text{"}
        & \text{"}
        \\ \hline
        \exIntro{1/1,2/1,3/3,3/2,2/3} & 
        \vphantom{\bigg|}
        \frac{z^6}{(1-z)^2(1-z^2)^3}
        &
        \begin{array}{c}
                \{1\}
            \end{array}
        &
            \begin{array}{c}
                1/(1-z^3)
            \end{array} 
        \\ \hline
        \exIntro{1/1,1/2,3/3,3/2} & 
        \vphantom{\bigg|}
        \frac{z^4}{(1-z)^2(1-z^2)^2}
        &
        \begin{array}{c}
                \{2\}
            \end{array}
        & \text{"}
        \\ \hline
        \exIntro{1/1,1/2,2/1,3/3,3/2} & 
        \vphantom{\bigg|}
        \frac{z^6}{(1-z)^2(1-z^2)^3}
        & \text{"}
        & \text{"}
        \\ \hline
        \exIntro{1/1,1/2,3/3,3/2,2/3} & 
        \vphantom{\bigg|}
        \frac{z^6}{(1-z)^2(1-z^2)^3}
        & \text{"}
        & \text{"}
        \\ \hline
        \exIntro{1/1,1/2,2/1,3/3,2/3} & 
        \vphantom{\bigg|}
        \frac{z^6}{(1-z)^2(1-z^2)^3}
        &
        \begin{array}{c}
                \{3\}
            \end{array}
        & \text{"}
        \\ \hline
        \exIntro{1/1,1/2,2/1,3/3,3/2,2/3} & 
        \vphantom{\bigg|}
        \frac{z^8}{(1-z)^2(1-z^2)^4}
        &
        \begin{array}{c}
                \emptyset
            \end{array}
        &
            \begin{array}{c}
                1
            \end{array} 
        \\ \hline
    \end{longtable}

Following Theorem~\ref{thm:Ehrhart in intro} (or Corollary~\ref{cor:Ehrhart series}), the Ehrhart series of $\M_{3,3}$ is obtained by summing over all the rows, where in each row, one takes the product of the second column with the sum of the rational functions in the fourth column.
Upon multiplying the result by the prefactor $1/(1-z^3)^3$, we obtain the power series expansion
\[
    \Ehr(\M_{3,3}; z) = 1 + 2 z + 7 z^2 + 18 z^3 + 41 z^4 + 86 z^5 + 176 z^6 + 325 z^7 + 
 587 z^8 + 1016 z^9 + 1686 z^{10} + \cdots,
\]
which we previewed in the introduction.

\bibliographystyle{amsplain}

\bibliography{references}

\end{document}